\documentclass[reqno]{amsart}
\usepackage{amssymb}
\usepackage{amsmath}
\usepackage{amsthm}
\usepackage{mathtools}
\usepackage[svgnames]{xcolor}
\usepackage[pdftex]{hyperref}
\usepackage{enumitem}
\usepackage{booktabs}
\usepackage{placeins }
\usepackage{soul}

\usepackage{tikz}
\usepackage{tkz-tab}
\usepackage{tkz-fct}  
\usetikzlibrary{calc,intersections}
\definecolor{color1}{RGB}{27,158,119}
\definecolor{color2}{RGB}{217,95,2}
\definecolor{color3}{RGB}{117,112,179}
\definecolor{color4}{RGB}{231,41,138}

\graphicspath{{./figures/}}

\usepackage{vmargin}
\setmarginsrb {2cm}{2cm}{2cm}{2cm} {1cm}{1cm}{0.5cm}{0.5cm}

\hypersetup{
  pdftitle={},
  pdfauthor={Perla Kfoury, Stefan Le Coz <slecoz@math.univ-toulouse.fr>, Tai-Peng Tsai},
  pdfsubject={},
  pdfkeywords={}, 
  colorlinks=true,
  linkcolor=DarkBlue  ,          
  citecolor=DarkRed,        
  filecolor=DarkMagenta,      
  urlcolor=DarkGreen,           
}

\newtheorem{theorem}{Theorem}[section]
\newtheorem{proposition}[theorem]{Proposition}
\newtheorem{lemma}[theorem]{Lemma}

\theoremstyle{remark}

\DeclarePairedDelimiter{\norm}{\lVert}{\rVert}

\newcommand{\eps}{\varepsilon}

\newcommand{\R}{\mathbb{R}}

\renewcommand{\leq}{\leqslant}
\renewcommand{\geq}{\geqslant}

\DeclareMathAlphabet{\mathpzc}{OT1}{pzc}{m}{it}

\usepackage{color}

\begin{document}

\title[Stability of standing waves of the double power 1d NLS]{Analysis of stability and instability for standing waves of the double power one dimensional nonlinear Schr\"odinger equation}

\author[P.~Kfoury]{Perla Kfoury}

\author[S.~Le Coz]{Stefan Le Coz}
\thanks{The work of P. K. and S. L. C. is 
  partially supported by ANR-11-LABX-0040-CIMI and CNRS IEA  296038}

\author[T.-P.~Tsai]{Tai-Peng Tsai}
\thanks{The work of T.-P. T. is partially supported by the NSERC grant RGPIN-2018-04137.}

\address[Perla Kfoury \& Stefan Le Coz]{Institut de Math\'ematiques de Toulouse ; UMR5219,
  \newline\indent
  Universit\'e de Toulouse ; CNRS,
  \newline\indent
  UPS IMT, F-31062 Toulouse Cedex 9,
  \newline\indent
  France}
\email[Perla Kfoury]{perla.kfoury@math.univ-toulouse.fr}
\email[Stefan Le Coz]{slecoz@math.univ-toulouse.fr}

\address[Tai-Peng Tsai]{
Department of Mathematics,
\newline\indent
University of British Columbia,
\newline\indent
Vancouver BC
\newline\indent
Canada V6T 1Z2}
\email[Tai-Peng Tsai]{ttsai@math.ubc.ca}


\date{\today}

\begin{abstract}
For the double power one dimensional nonlinear Schr\"odinger equation, we establish a complete classification of the stability or instability of standing waves with positive frequencies. In particular, we fill out the gaps left open by previous studies. Stability or instability follows from the analysis of the slope criterion of Grillakis, Shatah and Strauss. The main new ingredients in our approach are a reformulation of the slope and the explicit calculation of the slope value in the zero-frequency case. Our theoretical results are complemented with numerical experiments.
\end{abstract}

\maketitle


\section{Introduction}
\label{sec:introduction}

Consider the one dimensional nonlinear Schr\"odinger equation with double power nonlinearity
\begin{equation} \label{eq:nls}
i\partial_t u + \partial_x^2 u + a_p|u|^{p-1}u+a_q|u|^{q-1}u =0,
\end{equation}
where $u:\R_t \times \R_x \to \mathbb C$, $a_p,a_q \in\mathbb R\setminus\{0\}$ and  $1<p<q<\infty$. When $a_p<0$, $a_q>0$, we say that the nonlinearity is \emph{defocusing-focusing}, with analogous definitions for other possible signs combinations. 

Nonlinear Schr\"odinger equations appear in many areas of physics such as nonlinear optics (see e.g.~\cite{Ag07}) or Bose-Einstein condensation. The double power nonlinearity is an important example of the possible nonlinearities appearing in soliton theory (see e.g.~\cite{AkAnGr99}). Via gauge transformations, the double power nonlinearity is also connected with the derivative nonlinear Schr\"odinger equation (see e.g.~\cite{HaOz92,LeWu18,Ph21}). The double power nonlinearity is also a typical example of a nonlinearity breaking the scaling invariance of the pure power case, while still being relatively tractable, and it may be used to study phenomena in the absence of scaling symmetry  (see e.g.~\cite{LeMaRa16} for the construction of blowing-up solutions).

The Cauchy problem for \eqref{eq:nls} is well known (see \cite{Ca03} and the references therein) to be well-posed in the energy space $H^1(\mathbb R)$: for any $u_0\in H^1(\mathbb R)$, there exists a unique maximal solution $u\in C((-T_*,T^*),H^1(\R)) \cap C^1((-T_*,T^*),H^{-1}(\R))$ of \eqref{eq:nls} such that $u(t=0)=u_0$. Moreover, the energy $E$ and the mass $M$, defined by
\[
E(u)=\frac12\norm{u_x}_{L^2}^2-\frac{a_p}{p+1}\norm{u}_{L^{p+1}}^{p+1}-\frac{a_q}{q+1}\norm{u}_{L^{q+1}}^{q+1},\quad M(u)=\frac12\norm{u}_{L^2}^2,
 \] 
are conserved along the flow and the blow-up alternative holds (i.e.~if $T^*<\infty$ (resp. $T_*<\infty$), then $\lim_{t\to T^* \text{ (resp }-T_*{\text{)}}}\norm{u(t)}_{H^1}=\infty$). 

A \emph{standing
wave} is a solution of \eqref{eq:nls} of the form $u(t,x) = e^{i \omega t}\phi(x) $ for some
$\omega \in \R$ and a \emph{profile} $\phi \in C^2(\R)$, which then satisfies
\begin{equation} \label{eq:ode}
-\phi'' +\omega \phi -a_p|\phi|^{p-1}\phi -a_q|\phi|^{q-1}\phi=0.
\end{equation}
We only consider real-valued $\phi$ in this paper.
Define $\omega^*$ by
\[
\omega^*=\sup\left\{\omega\geq 0: \exists s>0 \text{ such that }\frac\omega2s^2 -\frac{a_p}{p+1}s^{p+1} -\frac{a_q}{q+1}s^{q+1}<0\right\}.
\]

It is well known (see \cite{BeLi83-1}) that  existence of non-trivial solutions of \eqref{eq:ode}
with $\lim_{|x|\to\infty}\phi(x)=0$ holds if and only if
\[
  \begin{cases}
    0\leq \omega <\omega^* & \text{when }a_p<0,a_q>0,\\
    0< \omega <\omega^* & \text{otherwise}.
  \end{cases}
\]
In that case, the solution is positive (up to phase shift), even (up to translation) and unique. We denote it by $\phi_\omega$, or simply $\phi$ when there is no ambiguity.

Solitary waves are the building blocks for the nonlinear dynamics of \eqref{eq:nls}, as it is expected that, generically, a solution of \eqref{eq:nls} will decompose into a dispersive linear part and a combination of nonlinear structures as solitary waves. This vague statement is usually referred to as the \emph{Soliton Resolution Conjecture}.

Therefore, understanding the dynamical properties of standing waves, in particular their stability, is a key step in the analysis of the dynamics of \eqref{eq:nls}.  Several stability concepts are available for standing waves. The most commonly used is \emph{orbital stability}, which is defined as follows. The standing wave $e^{i\omega t}\phi(x)$ solution of \eqref{eq:nls} is said to be \emph{orbitally stable} if the following holds. For any $\eps>0$ there exists $\delta>0$ such that if $u_0\in H^1(\R)$ verifies
\[
\norm{u_0-\phi}_{H^1}<\delta,
\]
then the associated solution $u$ of \eqref{eq:nls} exists globally and verifies 
\[
\sup_{t\in\R}\inf_{y\in\R,\theta\in\R}\norm*{u(t)-e^{i\theta}\phi(\cdot-y)}_{H^1}<\eps.
  \]
In the rest of this paper, when we talk about stability/instability, we always mean \emph{orbital} stability/instability.

  The groundwork for orbital stability studies was laid down by Berestycki and Cazenave \cite{BeCa81}, Cazenave and Lions \cite{CaLi82} and Weinstein \cite{We83,We85}. Two approaches lead to stability or instability results: the variational approach of \cite{BeCa81,CaLi82}, which exploits global variational characterizations combined with conservation laws or the virial identity, and the spectral approach of \cite{We83,We85}, which exploits spectral and coercivity properties of linearized operators to construct a suitable Lyapunov functional.  Later on, Grillakis, Shatah and Strauss \cite{GrShSt87,GrShSt90} developed an abstract theory which, under certain assumptions, boils down the stability study of a branch of standing waves $\omega\to \phi_{\omega}$  to the study of the sign of the quantity (usually called \emph{slope})
\(
\frac{\partial}{\partial \omega} M(\phi_{\omega}).
\)
Note that the theory of Grillakis, Shatah and Strauss has known recently a considerable revamping in the works of De Bièvre, Genoud and Rota-Nodari \cite{DeGeRo15,DeRo19}.

With the above mentioned techniques, the orbital stability of positive standing waves has been completely determined in the single power case (i.e.~$a_q=0$) in any dimension $d\geq1$ in \cite{BeCa81,CaLi82, We83,We85}. If $a_q=0$, positive standing waves exist if and only if $a_p>0$ and $\omega>0$. In this case, they are stable if $1<p<1+\frac4d$ (i.e.~$1<p<5$ in dimension $d=1$), and they are unstable if $1+\frac4d\leq p<1+\frac{4}{(d-2)_+}$ (i.e.~$5\leq p<\infty$ in dimension $d=1$). Scaling properties of the single power nonlinearity play an important role in the proof and ensure in particular that stability and instability are independent of the value of the frequency $\omega$. It turns out that there is no scaling invariance for double power nonlinearities, which makes the stability study more delicate. As a matter of fact, only very partial results are available so far in higher dimensions. In dimension $1$, the situation is a bit more favorable, as one might exploit the ODE structure of the profile equation \eqref{eq:ode} in the analysis.

Preliminary investigations for the stability of standing waves in dimension $1$ were conducted by Iliev and Kirchev \cite{IlKi93} in the case of a generic nonlinearity. In particular, a formula for the slope condition was obtained in \cite{IlKi93}. 
The earliest work devoted to the stability of standing waves for nonlinear Schr\"odinger equations with double power nonlinearity in dimension $1$ is the work of Ohta \cite{Oh95}. In this work, using the integral expression for the slope condition derived by Iliev and Kirchev \cite{IlKi93}, Ohta established the stability/instability of standing waves in a number of cases. Later on, Maeda \cite{Ma08} further refined the approach of Ohta and established the stability/instability in most of the situations not covered in \cite{Oh95}. However, the stability picture was still not complete, as the following case was left partially open:
\[
a_p<0,\quad a_q>0,\quad  1<p<q<5,\quad p+q\leq 6\text{ or }p\leq \frac73.
\]
In the above case, Ohta \cite{Oh95} established the stability of standing waves for $\omega$ large enough. The instability for small $\omega$ was obtained by Ohta \cite{Oh95} for $p+q> 6$, a condition which was later improved to
$(p+3)(q+3)>32$
by Fukaya and Hayashi \cite{FuHa21}.
What happens in the intermediate range of $\omega$ when 
\[
(p+3)(q+3)>32\quad \text{and}\quad \left(p+q\leq6 \text{ or } p\leq \frac73\right),
\]
was not elucidated in \cite{FuHa21, Ma08,Oh95}, nor what happens for small $\omega$ when $(p+3)(q+3)\leq32$,
(except for the notable case $p=2$, $q=3$, where explicit calculations  are possible  and show that the wave is stable for any $\omega>0$).

For convenience, we adopt the following convention. When a standing  wave is stable for any $\omega\in(0,\omega^*)$, we say that it is of type S. When there exists $\omega_1\in(0,\omega^*)$ such that the standing wave is unstable for $\omega\in(0,\omega_1]$ and stable for $\omega\in(\omega_1,\omega^*)$, we say that it is of type US.
Other types are defined in a similar manner. Note that when instability holds the endpoint $\omega_1$ is included in the instability range (thanks to the criterion of Comech and Pelinovsky \cite{CoPe03}, see \eqref{eq:comech-pelinovsky}).

Our goal in this paper is to fill out the gaps left open by the previous works \cite{FuHa21, Ma08,Oh95} and to provide a complete stability picture for the standing waves of the Schr\"odinger equations with double power nonlinearity.
Our main result is the following.

\begin{theorem}
  \label{thm:main}
  Let $(\phi_{\omega})_{\omega\in(0,\omega^*)}$ be the family of standing waves of \eqref{eq:nls}. The following gives the stability type of the family of standing waves.
  \begin{enumerate}
  \item Assume that $a_p>0$ and $a_q>0$.
    \begin{enumerate}
    \item If $q\leq 5$, then it is of type $S$.
    \item If $p\geq 5$, then it is of type $U$.
    \item If $p<5<q$, then it is of type $SU$.   
    \end{enumerate}
  \item Assume that $a_p>0$ and $a_q<0$.
    \begin{enumerate}
    \item If $p\leq 5$, then it is of type $S$.
    \item If $p> 5$, then it is of type $US$.
    \end{enumerate}
  \item Assume that $a_p<0$ and $a_q>0$.
       \begin{enumerate}
    \item If $q\leq 7-2p$, then it is of type $S$.   
    \item If $7-2p< q< 5$, then it is of type $US$.
       \item If $q\geq 5$, then it is of type $U$.
    \end{enumerate}
  \end{enumerate}
In particular, in the cases 1(c), 2(b) and 3(b) with stability change, the standing wave $\phi_{\omega_c}$ at the critical frequency $\omega_c$ is unstable.
\end{theorem}

This theorem implies in particular that stability change occurs at most once, which is conjectured in \cite[p.~265]{Ma08}, and is in contrast to NLS with triple power nonlinearity considered in \cite{LiTsZw21}.

In Theorem \ref{thm:main}, cases (1), (2) and (3)(c) were already covered in \cite{Ma08,Oh95}. For the sake of completeness, and as the proofs are not very long, we will also cover them in our work. Cases (3)(a) and (b) were only partially solved. We provide a definitive result for these cases. Our approach relies on several ingredients. First of all, we express the slope condition in a concise, while easily tractable integral, factoring out terms which are in any case positive. Instead of working with the parameter $\omega$, we manipulate the slope condition  with the parameter $\phi_0=\phi_\omega(0)$ (which is in a  bijective relation with $\omega$). We are left with an integral expression $F(\phi_0)$ (see \eqref{eq:F(phi_0)}), of which we need to determine the sign. A refactorisation allows us to introduce an auxiliary parameter $\gamma$, and differentiation with respect to $\phi_0$ gives us an expression which we can prove to have sign, provided we have suitably chosen the parameter $\gamma$. This gives the information that  $F(\phi_0)$ changes sign at most once. The sign for large $\omega$ (or equivalently large $\phi_0$) had already been established in \cite{Oh95}. On the other hand, the sign for $\omega$ close to $0$ had not been computed before. Here, an astute rewriting of the slope in terms of Beta functions allows us to determine the sign for $\omega$ close to $0$.

Observe that our results are not covering the zero-frequency case $\omega=0$. Stability or instability of the corresponding (algebraic) standing waves (when existing) can be conjectured to be the same as the one for small $\omega>0$ (which is consistent with the results obtained by Fukaya and Hayashi \cite{FuHa21}).

In the case $a_p<0$, $a_q>0$ and $1<p<q<5$, we complement our theoretical results with numerical experiments. We first represent the critical surface $\{\omega_c(p,q)\}$ at which the stability change occurs and discuss the different possible shapes of the surface depending on the ratio $a_p/a_q$. We then simulate the dynamics of \eqref{eq:nls} around a standing wave with the Crank-Nicolson scheme with relaxation of Besse \cite{Be04}. Three types of behaviors are observed depending on the type of initial data : stability, growth followed by oscillations, and scattering. 

To end this introduction, we point out that many works are  devoted to standing waves of  the double power nonlinear Schr\"odinger in higher dimension (for which our approach does not apply), and just give a small sample of the existing literature. The cubic quintic case in higher dimension was investigated in \cite{CaKlSp20}. Stability of standing waves in higher dimension for generic nonlinearities was considered in \cite{Fu03}. Strong instability was studied in 
\cite{OhYa15}. Stability results for algebraic standing waves were obtained in \cite{FuHa21}.
Uniqueness and non-degeneracy was considered in \cite{LeRo20}. Existence or non-existence of minimizers of the energy at fixed mass was obtained in \cite{BeFoGe20}.
Let us also mention in dimension $1$ the work \cite{GeMaWe16}, which is devoted to the stability of standing waves for cubic-quintic nonlinearities in the presence of a $\delta$ potential (see \cite{AnHe19,AnHePl19} for further developments).

This paper is organized as follows. We start by some preliminaries in Section \ref{sec:preliminaries}, recalling in particular the properties of the standing wave profiles and the stability criterion. In Section \ref{sec:slope}, we reformulate the slope condition for stability, using the profile equation. In Section \ref{sec:endpoints}, we analyze the limit of the slope at the endpoints of the interval of admissible frequencies $\omega$ and in particular determine the sign of the slope at the endpoints. The sign of the slope on the full interval of admissible frequencies is recovered in Section \ref{sec:sign}, which shows Theorem \ref{thm:main}. Finally, numerical experiments are presented in Section \ref{sec:numerics}.

After the first version of this paper was posted to arXiv, Professor Hayashi kindly informed us he had an independent similar result and posted it as \cite{Ha21}. His Theorem 1.3 is similar to our Theorem 1.1 although it does not include the case $1<p<9/5$.

\section{Preliminaries}
\label{sec:preliminaries}

\subsection{The profile equation}

We start by some analysis around the ordinary differential equation \eqref{eq:ode} and its solutions $(\phi_\omega)$.
Apart in a few specific cases (e.g.~when $q=2p-1$, see e.g.~\cite{LiTsZw21}), there does not exist an explicit formula for the full standing waves profile.
Note that $\omega^*=\infty$ when $a_q>0$, $0<\omega^*<\infty$ when $a_p>0$ and $a_q<0$, and $\omega^*=-\infty$ (i.e.~there is no solution of \eqref{eq:ode} in $H^1(\R)$) when $a_p,a_q<0$. All along this paper, we assume that $0<\omega<\omega^*$ (excluding in particular the possibility that $a_p,a_q<0$). Under this assumption, there exists $\phi_0>0$ (depending implicitly on $\omega$) such that
\[
\phi_0 =\inf \{ \phi>0: \frac{\omega}{2}\phi^2-\frac{a_p}{p+1}|\phi|^{p+1} -\frac{a_q}{q+1}|\phi|^{q+1}=0\},
\]
and we have
\[
\phi_\omega(0)=\phi_0.
\]
Observe that $\omega$ may be expressed in terms of $\phi_0$ as follows
\begin{equation}
  \label{eq:omega_in_function_of_phi_0}
\omega=\frac{2a_p}{p+1}\phi_0^{p-1} +\frac{2a_q}{q+1}\phi_0^{q-1}.
\end{equation}
Moreover, as $\omega<\omega^*$, we have
\[
  \omega -a_p\phi_0^{p-1} -a_q\phi_0^{q-1}<0.
\]
This implies in particular $\phi_0$  is a $C^1$-function of $\omega$.
Moreover, we always have
  \begin{equation}
    \label{eq:dphidomega}
\frac{\partial \phi_0}{\partial\omega}=\left(\frac{2a_p(p-1)}{p+1}\phi_0^{p-2} +\frac{2a_q(q-1)}{q+1}\phi_0^{q-2}\right)^{-1}>0.
\end{equation}
As a consequence, the following result holds.

\begin{lemma}
  The function $\omega\to \phi_0$ is a strictly increasing bijection from $(0,\omega^*)$ to $(\phi_*,\phi^*)$ where
  \begin{equation}
    \label{eq:formula_of_phi*}
    \phi_*=
    \begin{cases}
      \left( -\frac{a_p}{a_q}{\frac {q+1}{p+1}}\right)^{\frac1{q-p}} &\text{ if }a_p<0,\\
      0 &\text{ if }a_p>0,
    \end{cases}
    \quad
    \phi^*=
    \begin{cases}
      \infty &\text{ if }a_q>0,\\
      \left( -\frac{a_p}{a_q}\frac{p-1}{q-1}{\frac {q+1}{p+1}}\right)^{\frac1{q-p}} &\text{ if }a_q<0.
    \end{cases}
  \end{equation}
\end{lemma}

\subsection{The stability criterion}
\label{sec:stab-crit}

As we already mentioned, stability criteria have been derived in the general case in \cite{GrShSt87,IlKi93}. For the double power nonlinearity, the stability of the standing wave is determined by a slope condition (the spectral condition of  \cite{GrShSt87} being always verified in this case when $\omega>0$).
The standing wave $e^{i\omega t}\phi_\omega(x)$ will be stable if
\[
  \frac{\partial}{\partial \omega}M(\phi_\omega)>0,
\]
and it will be unstable if 
\[
    \frac{\partial}{\partial \omega}M(\phi_\omega)<0.
\]
When $  \partial_\omega M(\phi_\omega)=0$, the stability can be decided by looking at the second derivative, as was established by Comech and Pelinovsky \cite{CoPe03}: If  $  \partial_\omega M(\phi_\omega)=0$ and
\begin{equation}
 \label{eq:comech-pelinovsky}
    \frac{\partial^2}{\partial \omega^2}M(\phi_\omega)\neq 0,
\end{equation}
  then the standing  wave $e^{i\omega t}\phi_\omega(x)$ is unstable.

  \section{Reformulation of the slope}
\label{sec:slope}

For notational convenience, we introduce the function $J$ defined by
\[
J(\omega,p,q)=  \frac{\partial}{\partial \omega}M(\phi_\omega).
\]
Hence the sign of $J$ determines the stability of the corresponding standing wave. 


The main idea in this section is to express $J$ in terms of $\phi_0$ instead of $\omega$. Before doing that, we introduce some convenient notation. 
Let $\Phi_p$ and $\Phi_q$ be defined by
\begin{equation} \label{eq:phi_p,phi_q}
\Phi_p=\frac{2a_p}{p+1}\phi_0^{p+1}(1-s^{p-1}), \quad \Phi_q=\frac{2a_q}{q+1}\phi_0^{q+1}(1-s^{q-1}),
\end{equation}
where $a_p,a_q \neq 0,$ $1<p<q<\infty$ and $0<s<1.$
\begin{lemma}
  \label{lem:J=CF}
The function $J$ may be expressed in terms of $\phi_0$ as follows
\[
J(\omega,p,q)=C(\phi_0)F(\phi_0),
\]
where 
\begin{equation} \label{eq:F(phi_0)}
F(\phi_0)=\int_0^1 \frac{(5-p)\Phi_p+(5-q)\Phi_q}{(\Phi_p+\Phi_q)^{\frac{3}{2}}}s ds,
\end{equation}
and $C(\phi_0)$ is positive and explicitly known (see \eqref{newC}).
\end{lemma}

\begin{proof}
We multiply the equation \eqref{eq:ode}  of the profile by $\phi_x$ and we integrate to obtain
\[
-\frac{1}{2} |\phi_x|^2+\frac{\omega}{2} |\phi|^2 -\frac{a_p}{p+1}|\phi|^{p+1}-\frac{a_q}{q+1}|\phi|^{q+1}=c.
\]
When $x \rightarrow \infty $, we know that $\phi(x) \rightarrow 0$ and $\phi'(x) \rightarrow 0$.
Therefore $c=0$, and
\begin{equation} \label{eq:firstintegral}
-\frac{1}{2} |\phi_x|^2+\frac{\omega}{2} |\phi|^2 -\frac{a_p}{p+1}|\phi|^{p+1}-\frac{a_q}{q+1}|\phi|^{q+1}=0.
\end{equation}
For $x>0$, as $\phi$ is decreasing, from \eqref{eq:firstintegral} we have 
\[
\phi_x=-\sqrt{\omega\phi^2-\frac{2a_p}{p+1}\phi^{p+1}-\frac{2a_q}{q+1}\phi^{q+1}}.
\]
Still for $x>0$, let $t=\phi(x)$, then
\[
dx=\frac{dt}{\phi_x}=-\frac{dt}{\sqrt{\omega\phi^2-\frac{2a_p}{p+1}\phi^{p+1}-\frac{2a_q}{q+1}\phi^{q+1}}}.
\]
Therefore we may perform the following change of variable:
  \begin{equation*}
    M(\phi)=\frac{1}{2} \int_{\R} |\phi(x)|^2 dx=\int_0^{\infty} |\phi(x)|^2 dx =\int_0^{\phi_0}\frac{t^2}{\sqrt{\omega t^2-\frac{2a_p}{p+1}t^{p+1}-\frac{2a_q}{q+1}t^{q+1}}}dt.
  \end{equation*}
Changing again variable by setting $t=\phi_0 s$, we have
\[
    M(\phi)=\int_0^1\frac{\phi_0^3 s^2}{s\sqrt{\omega\phi_0^2-\frac{2a_p}{p+1}\phi_0^{p+1}s^{p-1}-\frac{2a_q}{q+1}\phi_0^{q+1}s^{q-1}}}ds.
\]
Replacing $\omega$ by its value \eqref{eq:omega_in_function_of_phi_0} in terms of $\phi_0$, we have
\[
    M(\phi)=\int_0^1\frac{\phi_0^3 s}{\sqrt{\frac{2a_p}{p+1}\phi_0^{p+1}+\frac{2a_q}{q+1}\phi_0^{q+1}-\frac{2a_p}{p+1}\phi_0^{p+1}s^{p-1}-\frac{2a_q}{q+1}\phi_0^{q+1}s^{q-1}}}ds,
\]
which, using the notation \eqref{eq:phi_p,phi_q} for $\Phi_p$ and $\Phi_q$, gives 
\[
M(\phi)=\int_0^1\frac{\phi_0^3 s}{\sqrt{\Phi_p+\Phi_q}}ds.
\]
Differentiating with respect to $\omega$, we have
\[
\partial_{\omega}\Phi_p=(p+1)\Phi_p \phi_0^{-1} \partial_{\omega}\phi_0,\quad \partial_{\omega}\Phi_q=(q+1)\Phi_q \phi_0^{-1 }\partial_{\omega}\phi_0.
\]
Therefore we obtain
\begin{align*}
J(\omega,p,q)=\partial_\omega M(\phi)&=\int_0^1 \frac{3\phi_0^2\partial_\omega\phi_0 s (\Phi_p+\Phi_q) -\frac12 \phi_0^3 s \left( (p+1)\Phi_p \phi_0^{-1}+(q+1)\Phi_q\phi_0^{-1})\right) \partial_\omega\phi_0 }{(\Phi_p+\Phi_q)^{\frac{3}{2}}}ds,\\
&= \frac{\partial_\omega\phi_0}{2}\phi_0^2 \int_0^1 \frac{6(\Phi_p+\Phi_q)-\left( (p+1)\Phi_p+(q+1)\Phi_q)\right)}{(\Phi_p+\Phi_q)^{\frac{3}{2}}}s ds,\\
&= \frac{\partial_\omega\phi_0}{2}\phi_0^2 \int_0^1 \frac{(5-p)\Phi_p+(5-q)\Phi_q}{(\Phi_p+\Phi_q)^{\frac{3}{2}}}s ds,\\
&=C(\phi_0)F(\phi_0),
\end{align*}
where $F(\phi_0)$ is defined in \eqref{eq:F(phi_0)} and 
\begin{equation} \label{newC}
C(\phi_0)= \frac{\partial_\omega\phi_0}{2}\phi_0^2.
\end{equation}
This concludes the proof.
\end{proof} 




We will now analyze the variations of $J(\omega,p,q)$ in terms of $\phi_0$. For future convenience (the reason for such a choice will appear clearly later), we introduce an auxiliary parameter $\gamma$ in the following way
 \[
J(\omega,p,q)=C(\phi_0)\phi_0^{-\gamma} F_{\gamma}(\phi_0),
\]
where 
\[
F_{\gamma}(\phi_0)=\int_0^1\phi_0^\gamma\left( \frac{(5-p)\Phi_p+(5-q)\Phi_q }{\left(   \Phi_p+\Phi_q
       \right)^{\frac32}}\right )sds.
\]
Denote the integrand of $F_{\gamma}$ by
  \begin{equation}
    \label{eq:def-I-gamma}
I_{\gamma}(\phi_0)=\phi_0^\gamma\left( \frac{(5-p)\Phi_p+(5-q)\Phi_q }{\left(   \Phi_p+\Phi_q \right)^{\frac32}}\right ).
\end{equation}
Observe that there is a implicit dependency in $s$. In the following lemma we differentiate $I_{\gamma}(\phi_0)$ with respect to $\phi_0$.

\begin{lemma} \label{lem:derivative_of_I}
For any $0<s<1$, the following holds:
\[
\frac{\partial I_{\gamma}}{\partial_{\phi_0}}= \frac12\phi_0^{\gamma-1}
 \left( \frac{\left((5-p)(2\gamma-(p+1))\Phi_p+(5-q)(2\gamma-(q+1))\Phi_q\right)(\Phi_p+\Phi_q)-3(q-p)^2\Phi_p\Phi_q
       }{\left(   \Phi_p+\Phi_q
         \right)^{\frac52}}\right ).
\]
\end{lemma}
\begin{proof}
We start by differentiating the term in parenthesis in $I_{\gamma}(\phi_0)$. We have
\[
\partial_{\phi_0}\Phi_p=(p+1)\Phi_p \phi_0^{-1},\quad \partial_{\phi_0}\Phi_q=(q+1)\Phi_q \phi_0^{-1 }.
\]
Therefore, we have
\begin{multline*}
\partial_{\phi_0}\left( \frac{(5-p)\Phi_p+(5-q)\Phi_q }{\left(   \Phi_p+\Phi_q \right)^{\frac32}}\right ), \\ =\phi_0^{-1}\left( \frac{\left((5-p)(p+1)\Phi_p+(5-q)(q+1)\Phi_q\right) \left(\Phi_p+\Phi_q \right)-\frac32 \left((5-p)\Phi_p+(5-q)\Phi_q\right)\left( (p+1)\Phi_p+(q+1)\Phi_q \right) }{\left(   \Phi_p+\Phi_q \right)^{\frac52}}\right ), \\
=\frac12 \phi_0^{-1}\left( \frac{-(5-p)(p+1)\Phi_p^2-(5-q)(q+1)\Phi_q^2 +((5-p)(2p-3q-1)+(5-q)(2q-3p-1))\Phi_p\Phi_q}{\left(   \Phi_p+\Phi_q \right)^{\frac52}}\right ),\\
=\frac12\phi_0^{-1}\left( \frac{-\left((5-p)(p+1)\Phi_p+(5-q)(q+1)\Phi_q\right)(\Phi_p+\Phi_q)-3(q-p)^2\Phi_p\Phi_q }{\left(   \Phi_p+\Phi_q \right)^{\frac52}}\right ).
\end{multline*}
Before going on, observe that we may rewrite the term in parentheses in $I_{\gamma}(\phi_0)$ as
\[
\frac{(5-p)\Phi_p+(5-q)\Phi_q }{\left(\Phi_p+\Phi_q \right)^{\frac32}}=
 \frac{2\left((5-p)\Phi_p+(5-q)\Phi_q \right) \left(   \Phi_p+\Phi_q \right)}{ 2\left(   \Phi_p+\Phi_q \right)^{\frac52}}.
\]
Finally, the full derivative of $I_{\gamma}(\phi_0)$ is given by
\begin{multline*}
\frac{\partial I_{\gamma}}{\partial \phi_0} =\partial_{\phi_0}\left(\phi_0^\gamma\left( \frac{(5-p)\Phi_p+(5-q)\Phi_q }{\left(   \Phi_p+\Phi_q
       \right)^{\frac32}}\right)\right),\\
=\frac12\phi_0^{\gamma-1}
 \left( \frac{\left((5-p)(2\gamma-(p+1))\Phi_p+(5-q)(2\gamma-(q+1))\Phi_q\right)(\Phi_p+\Phi_q)-3(p-q)^2\Phi_p\Phi_q
       }{\left(   \Phi_p+\Phi_q
         \right)^{\frac52}}\right ).\\
\end{multline*}
This concludes the proof.
\end{proof} 

For future reference, we establish here the following technical lemma which we will use at several occasions.
\begin{lemma}\label{lem:f_bijective}
The function $s \rightarrow \frac{1-s^{q-1}}{1-s^{p-1}}$ is an increasing bijection from $(0,1)$ to $(1,\frac{q-1}{p-1})$.
\end{lemma}
\begin{proof}
Let $h(s)=\frac{1-s^{q-1}}{1-s^{p-1}}$.
We have 
\[
h'(s) = \frac{s^{p-2}}{(1-s^{p-1})^2}l(s),
\]
where
\[
l(s) = (q-p) s^{q-1} + p-1 - (q-1)s^{q-p}.
\]
Note that $l(1)=0$ and for $0<s<1$,
\[
l'(s) = (q-p)(q-1)( s^{q-2} - s^{q-p-1})<0.
\]
Hence $l'(s)<0$ and $l(s)>0$ for $0<s<1$. We conclude that $h'(s)>0$ for $0<s<1$.
As a consequence, $h$ is increasing on the interval $(0,1)$. Moreover, we have $h(0)=1$ and, by L'Hospital's rule,
\[
\lim_{s \to 1} h(s) = \frac{q-1}{p-1}.
\]
This concludes the proof.
\end{proof}

\section{The slope at the endpoints}
\label{sec:endpoints}

Our goal in this section is to investigate what happens for $J(\omega,p,q)$ when $\omega$ is close to $0$ and $\omega^*$. 

\subsection{The zero frequency case}

In this section, we determine the limit of $J(\omega,p,q)$ when $\omega$ tends to zero.
Let $J_0$ be defined by 
\[
J_0(p,q)=\lim_{\omega \to 0} J(\omega,p,q).
\]

We first consider the case where $a_p>0$.


\begin{proposition}
  \label{prop:sign_J_0_a_p_positive}
  Let $a_p>0$. The following holds.
  \begin{enumerate}
  \item If $1<p<\frac{7}{3}$, then $J_0(p,q)=0^+$.
  \item If $p=\frac{7}{3}$, then $0<J_0(p,q)<\infty$.
  \item If $\frac{7}{3}<p<5$, then $J_0(p,q)=\infty$.
  \item If $p=5$, then three cases have to be distinguished.
    \begin{enumerate}
    \item If $q<9$, then $J_0(p,q)=-\operatorname{sign}(a_q)\infty$.
    \item If $q=9$, then $0<-\operatorname{sign}(a_q) J_0(p,q)<\infty$.
    \item If $q>9$, then $J_0(p,q)=0^{-\operatorname{sign}(a_q)}$.
    \end{enumerate}
  \item If $5<p$, then $J_0(p,q)=-\infty$.
  \end{enumerate}
\end{proposition}

\begin{proof}
  When $a_p>0$, we have
  \[
\lim_{\omega\to0}\phi_0=\phi_*=0.
\]
Recall that we have shown in Lemma \ref{lem:J=CF} that $J$ may be written as $J(\omega,p,q)=C(\phi_0)F(\phi_0)$. We have (recalling the definition \eqref{newC} of $C$ and the expression \eqref{eq:dphidomega} of $\partial_\omega\phi_0$)
  \begin{multline*}
    C(\phi_0)=\frac12\partial_\omega\phi_0\phi_0^2=\frac14\left(\frac{a_p(p-1)}{p+1}\phi_0^{p-4} +\frac{a_q(q-1)}{q+1}\phi_0^{q-4}\right)^{-1}\\=\phi_0^{4-p}\frac14\left(\frac{a_p(p-1)}{p+1} +\frac{a_q(q-1)}{q+1}\phi_0^{q-p}\right)^{-1}=\phi_0^{4-p}\left(\frac {p+1}{4a_p(p-1)} +o(1)\right).
  \end{multline*}
The function $F$ (defined in \eqref{eq:F(phi_0)}) can be written, substituting $\Phi_p$ and $\Phi_q$ by their expressions \eqref{eq:phi_p,phi_q}, as
\[
  F(\phi_0)=\int_0^1\frac{\frac{2a_p(5-p)}{p+1}(1-s^{p-1})\phi_0^{p+1}+\frac{2a_q(5-q)}{q+1}(1-s^{q-1})\phi_0^{q+1}}{
\left(\frac{2a_p}{p+1}(1-s^{p-1})\phi_0^{p+1}+\frac{2a_q}{q+1}(1-s^{q-1})\phi_0^{q+1}\right)^{\frac32}
  }sds.
\]
As we are interested in the limit $\phi_0\to 0$, we factor out the terms in $\phi_0^{p+1}$ to get
  \begin{multline*}
    F(\phi_0)=\phi_0^{-\frac{p+1}2}\int_0^1\frac{\frac{2a_p(5-p)}{p+1}(1-s^{p-1})+\frac{2a_q(5-q)}{q+1}(1-s^{q-1})\phi_0^{q-p}}{ \left(\frac{2a_p}{p+1}(1-s^{p-1})+\frac{2a_q}{q+1}(1-s^{q-1})\phi_0^{q-p}\right)^{\frac32} }sds
    \\
   = (5-p)\phi_0^{-\frac{p+1}2}\left(\int_0^1 \left(\frac{2a_p}{p+1}(1-s^{p-1})\right)^{-\frac12} sds+o(1)\right).
 \end{multline*}
 In the particular case $p=5$, we instead write
 \[
F(\phi_0)=a_q(5-q) \phi_0^{q-8}\left(\int_0^1 \frac{\frac{2}{q+1} (1-s^{q-1})}{\left(\frac{a_p}{3}(1-s^{4})\right)^{\frac32} } sds+o(1)\right).
   \]
In summary, when $\phi_0\to0$ (i.e.~$\omega\to0$), we have established that there exists $C=C(p,q)>0$ such that when $p\neq 5$ we have 
\[
J(\omega,p,q)=(5-p)\phi_0^{\frac{7-3p}{2}}C(1+o(1)),
\]
and when $p=5$ we have
\[
J(\omega,p,q)=(5-q)a_q\phi_0^{q-9}C(1+o(1)).
\]
This gives the desired result. 
\end{proof}


We now discuss the case $a_p<0$ and $a_q>0$.

\begin{proposition} \label{propJ0}
Let $a_p<0$ and $a_q>0$. 
\begin{enumerate}
   \item Assume that $p<\frac{7}{3}$. Then $J_0(p,q)\in \mathbb R$ and the following holds. 
    \begin{enumerate}
        \item If $2p+q<7$, then $J_0(p,q)>0$.
        \item  If $2p+q=7$, then $J_0(p,q)=0$.
        \item If $2p+q>7$, then $J_0(p,q)<0$.
    \end{enumerate}
  \item Assume that $p\geq \frac{7}{3}$. Then   $J_0(p,q)=-\infty$.
\end{enumerate}
\end{proposition}

We start with some preliminaries. To establish the first part of Proposition \ref{propJ0}, we will calculate $J_0$ in terms of the Beta function. Recall that the Beta function, also called Euler integral of the first kind, is a special function closely related to the Gamma function. It is defined for $x>0$ and $y>0$ by the integral
\begin{equation} \label{betafunction}
B(x,y)=\int_0^1 t^{x-1}(1-t)^{y-1}dt.
\end{equation}
The relation between the Beta function and the Gamma function is given by (see e.g \cite{AbSt64})
\begin{equation*}
B(x,y)=\frac{\Gamma(x)\Gamma(y)}{\Gamma(x+y)}.
\end{equation*}
We introduce the function $H$ defined for $x>0$ and $y>0$ by 
\begin{equation} \label{Hintegrale}
H(x,y)= \int_0^1 \frac{t^{x-1}(1-t^y)}{{(1-t)}^{\frac{3}{2}}}dt.
\end{equation}
The relation between $H$ and $B$ is given in the following lemma.

\begin{lemma} 
For $x>0$ and $y>0$, we have
\begin{equation} \label{Hformula}
H(x,y)=-(2x-1)B\left(x,\frac{1}{2}\right) + (2x+2y-1)B\left(x+y, \frac{1}{2}\right).
\end{equation}
\end{lemma}

\begin{proof}
Let 
\[
u(t)=t^{x-1}(1-t^y).
\]
Rewrite 
\[
\frac{1}{(1-t)^{\frac{3}{2}}}= v'(t)-\frac{1}{(1-t)^{\frac{1}{2}}},
\]
where 
\[
v(t)=\frac{2}{(1-t)^\frac{1}{2}}-2(1-t)^{\frac{1}{2}}=\frac{2t}{(1-t)^\frac{1}{2}}.
\]
We have
\begin{align*}
H(x,y)=&\int_0^1u(t)\left(v'(t)-\frac{1}{(1-t)^{\frac{1}{2}}}\right)dt,\\
=& u(1_{-})v(1_{-})-u(0_{+})v(0_{+})-\int_0^1u'(t)v(t)dt -\int_0^1 \frac{u(t)}{(1-t)^{\frac{1}{2}}}dt,\\
=& 0 -\int_0^1 \frac{2tu'(t)+u(t)}{(1-t)^{\frac{1}{2}}}dt.
\end{align*}
Above we have used $u(1)=0$ of order 1 to cancel the singularity of $v(1_-)$ of order $\frac{1}{2}$, and $v(0)=0$ with order $1$ to cancel the singularity of $u(0_+)$ of order $x-1$. 
Note that
\[
2tu'(t)+u(t)=(2x-1)t^{x-1}-(2x+2y-1)t^{x+y-1}.
\]
Therefore, using the definition of $B$ given in \eqref{betafunction} with $y=\frac{1}{2}$, we have
\[
H(x,y)=-(2x-1)B\left(x,\frac{1}{2}\right) + (2x+2y-1)B\left(x+y, \frac{1}{2}\right).
\]
This concludes the proof.
\end{proof}

The value $J_0(p,q)$ may be expressed using $B$ as follows.

\begin{lemma} \label{lemJ0sign}
Let $a_p<0$ and $a_q>0$. Assume that $1 < p < 7/3$. Then
\[
J_0(p,q) = (7-2p-q)C_0 B\left(\frac{7-3p}{2(q-p)}, \frac{
1}{2}\right),
\]
where $C_0$ is a positive constant explicitly known (given by \eqref{C0formula}).
\end{lemma}

The first part of Proposition \ref{propJ0} is a direct consequence of Lemma \ref{lemJ0sign}.

\begin{proof}[Proof of Lemma \ref{lemJ0sign}]
  Let $1<p<7/3$. 
   Recall that $J(\omega,p,q)=C(\phi_0)F(\phi_0)$, with $C(\phi_0)>0$ and $F$ given by \eqref{eq:F(phi_0)}. 
   Observe that, using the value of $\phi_*$ given in \eqref{eq:formula_of_phi*}, we may introduce the constant
\[
C_*=\frac{2a_q}{q+1}\phi_*^{q+1}=-\frac{2a_p}{(p+1)}\phi_*^{p+1}.
\]
  Using the definition \eqref{newC} of $C(\phi_0)$ and the expression \eqref{eq:dphidomega} of $\partial_\omega\phi_0$, we have
  \[
    \begin{aligned}
      \lim_{\phi_0\to \phi_*}C(\phi_0)&=C(\phi_*)=\frac{\phi_*^5}{2C_*(q-p)}>0,\\
      \lim_{\phi_0\to \phi_*}\Phi_p&=\frac{2a_p}{(p+1)}\phi_*^{p+1}(1-s^{p-1})=-C_*(1-s^{p-1}),\\
      \lim_{\phi_0\to \phi_*}\Phi_q&=\frac{2a_q}{(q+1)}\phi_*^{q+1}(1-s^{q-1})=C_*(1-s^{q-1}).
    \end{aligned}
  \]
  As a consequence, we get
\[
  \begin{aligned}
    \lim_{\phi_0\to \phi_*}(\Phi_p+\Phi_q)&=
    C_*(s^{p-1}-s^{q-1}),\\
      \lim_{\phi_0\to \phi_*}((5-p)\Phi_p+(5-q)\Phi_q)&=
      C_*\left(-(5-p)(1-s^{p-1})+
      (5-q)(1-s^{q-1})\right).
  \end{aligned}
\]
As a consequence,
\begin{equation}
  \label{eq:F_of_phi_*}
\begin{multlined}    F(\phi_*)=C_*^{-\frac12}\int_0^1\frac{-(5-p)(1-s^{p-1})+(5-q)(1-s^{q-1})}{\left(s^{p-1}-s^{q-1}\right)^{\frac32}}sds
    \\= C_*^{-\frac12}\int_0^1\frac{-(q-p)(1-s^{p-1})+(5-q)(s^{p-1}-s^{q-1})}{\left(1-s^{q-p}\right)^{\frac32}} s^{\frac{5-3p}{2}}ds
    \\= C_*^{-\frac12}\left(-(q-p)\int_0^1\frac{(1-s^{p-1})s^{\frac{5-3p}{2}}}{\left(1-s^{q-p}\right)^{\frac32}} ds
      +(5-q)\int_0^1s^{\frac{3-p}{2}}\left(1-s^{q-p}\right)^{-\frac12}ds\right).
  \end{multlined}
\end{equation}
Changing variable $t=s^{q-p}$, we obtain
\[
  F(\phi_*)=   C_*^{-\frac12}\left(-\int_0^1\frac{(1-t^{\frac{p-1}{q-p}})t^{\frac{7-p-2q}{2(q-p)}}}{\left(1-t\right)^{\frac32}} ds
      +\frac{5-q}{q-p}\int_0^1t^{\frac{5+p-2q}{2(q-p)}}\left(1-t\right)^{-\frac12}ds\right).
\]
We now use $B$ and $H$ to express the above quantity. Setting
      \begin{align*}
        (x_1,y_1)&=\left(\frac{7-p-2q}{2(q-p)}+1,\frac{p-1}{q-p}\right)=\left(\frac{7-3p}{2(q-p)}, \frac{p-1}{q-p}\right),\\ 
        (x_2,y_2)&=\left(\frac{5+p-2q}{2(q-p)}+1, \frac{1}{2}\right)=\left(\frac{5-p}{2(q-p)}, \frac{1}{2}\right),
      \end{align*}
      we get
\[
   F(\phi_*)=   C_*^{-\frac12}\left(-H(x_1,y_1)
      +\frac{5-q}{q-p}B(x_2,y_2)\right).
   \]
   Observe that we have assumed $p<\frac{7}{3}$, $p<q$, which ensures that $x_1,x_2,y_1,y_2$ are positive. This a posteriori justifies the fact that $J_0(p,q)$ is finite. The formula \eqref{Hformula} allows us to express $H(x_1,y_1)$ in the following way (using $y_2=1/2$):
   \[
H(x_1,y_1)=-(2x_1-1)B\left(x_1,y_2\right)+(2x_1+2y_1-1)B\left(x_1+y_1,y_2\right).
\]
It turns out that
\[
-(2x_1-1)=-\frac{7-2p-q}{q-p},\quad (2x_1+2y_1-1)=\frac{5-q}{q-p},\quad x_1+y_1=\frac{5-p}{2(q-p)}=x_2.
\]
As a consequence, there is a simplification in the expression of $F(\phi_*)$, which becomes
\[
   F(\phi_*)=   C_*^{-\frac12}\frac{7-2p-q}{q-p} B\left(x_1,y_2\right).
 \]
 Setting
   \begin{equation}
     \label{C0formula}
C_0=\frac{ C_*^{-\frac12}}{q-p}C(\phi_*)>0
\end{equation}
gives the desired result. 
\end{proof}

\begin{lemma} \label{lemJ0=infty}
Assume that  $a_p<0$ and $a_q>0$. For $p \geq 7/3$ and $1<p<q$, we have 
\[
\lim_{\omega \to 0} J(\omega,p,q)=-\infty.
\]
\end{lemma}
The second part of Proposition \ref{propJ0} is a direct consequence of Lemma \ref{lemJ0=infty}.

\begin{proof}
Coming back to the expression \eqref{eq:F_of_phi_*} of $F(\phi_*)$ in the proof of Lemma \ref{lemJ0sign},   we observe that if $\frac{5-3p}{2}\leq -1$, i.e.~$p\geq \frac73$, then $ F(\phi_*)=-\infty$, and, since $\lim_{\phi_0\to\phi_*}C(\phi_0)=C(\phi_*)>0$, we also have $J_0(p,q)=-\infty$ when $p\geq \frac73$.
\end{proof}

\subsection{The large frequency case}

In this section, we determine the limit of $J(\omega,p,q)$ when $\omega$ tends to $\omega^*$.
Let $J^*$ be defined by 
\[
J^*(p,q)=\lim_{\omega \to \omega^*} J(\omega,p,q).
\]

We first consider the case where $a_q>0$.

\begin{proposition}
  \label{prop:sign_J^*_a_q_positive}
  Let $a_q>0$. The following holds.
  \begin{enumerate}
  \item If $1<q<\frac{7}{3}$, then $J^*(p,q)=0^+$.
  \item If $q=\frac{7}{3}$, then $0<J^*(p,q)<\infty$.
  \item If $\frac{7}{3}<q<5$, then $J^*(p,q)=\infty$.
  \item If $q=5$, then $J^*(p,q)=0^{\operatorname{sign}(a_p)}$.
  \item If $5<q$, then $J^*(p,q)=-\infty$.
  \end{enumerate}
\end{proposition}

\begin{proof}
  Since $a_q>0$, we have $\omega^*=\infty$ and therefore $\phi^*=\infty$.
  Following similar arguments as in the proof of Proposition \ref{prop:sign_J_0_a_p_positive}, as $\phi_0\to\infty$, for $q\neq 5$, we have
  \begin{align*}
    C(\phi_0)&=\phi_0^{4-q}\left(\frac {q+1}{4a_q(q-1)} +o(1)\right),\\
    F(\phi_0)&=   (5-q)\phi_0^{-\frac{q+1}2}\left(\int_0^1 \left(\frac{2a_q}{q+1}(1-s^{q-1})\right)^{-\frac12} sds+o(1)\right).
  \end{align*}
As a consequence, for $q\neq 5$, when $\phi_0\to\infty$ (i.e.~$\omega\to\infty$), there exists $C=C(a_q,q)>0$ such that
\[
J(\omega,p,q)=(5-q)\phi_0^{\frac{7-3q}{2}}C(1+o(1)).
\]
In the particular case $q=5$, we instead write
   \begin{align*}
F(\phi_0)&=a_p(5-p) \phi_0^{p-8}\left(\int_0^1 \frac{\frac{2}{p+1} (1-s^{p-1})}{\left(\frac{a_q}{3}(1-s^{4})\right)^{\frac32} } sds+o(1)\right).
      \end{align*}
   and therefore we get
\[
J(\omega,p,q)=a_p(5-p)\phi_0^{p-9}C(1+o(1)).
\]
The two estimates on $J$ lead to the desired result.
\end{proof}

Then we consider the case where $a_q<0$ (and thus $a_p>0$ to ensure existence of standing waves).

\begin{proposition}
  \label{prop:sign_J^*_a_q_negative}
  Let $a_p>0$, $a_q<0$ and $5\leq p$. Then
  \[
J^*(p,q)=\infty.
    \]
\end{proposition}

Proposition \ref{prop:sign_J^*_a_q_negative} does not cover the whole possible range of $p$ and $q$. As it was not necessary in our analysis, we  did not try to cover the remaining cases. 

\begin{proof}[Proof of Proposition \ref{prop:sign_J^*_a_q_negative}]
  By construction, $\omega_*=\omega(\phi_*)$ is the value of $\omega$ at which $\partial_{\phi_0}\omega(\phi^*)=0$. As a consequence, we have
  \[
    \lim_{\phi_0\to\phi^*}\frac{\partial\phi_0}{\partial\omega}=\infty,
  \]
  which, given the value \eqref{newC} of $C(\phi_0)$, readily implies
  \[
    \lim_{\phi_0\to\phi^*}C(\phi_0)=\infty.
    \]
   Using the expressions of \eqref{eq:phi_p,phi_q} of $\Phi_p$ and $\Phi_q$ and the expression \eqref{eq:F(phi_0)} of $F$ we have
    \[
F(\phi^*)=\int_0^s\frac{2a_p\frac{(5-p)(\phi^*)^{p+1}(1-s^{p-1})}{p+1}+2a_q\frac{(5-q)(\phi^*)^{q+1}(1-s^{q-1})}{q+1}}{\left(2a_p\frac{(\phi^*)^{p+1}(1-s^{p-1})}{p+1}+2a_q\frac{(\phi^*)^{q+1}(1-s^{q-1})}{q+1}\right)^{\frac32}}sds.
\]
If $p=5$, then we have $F(\phi_*)>0$ and the conclusion follows. From now on, assume that $p> 5$.
Recalling the value of $\phi^*$ given in \eqref{eq:formula_of_phi*}, we infer that 
\begin{multline*}
  2a_p\frac{(5-p)(\phi^*)^{p+1}(1-s^{p-1})}{p+1}+2a_q\frac{(5-q)(\phi^*)^{q+1}(1-s^{q-1})}{q+1}
  \\
  =2a_q (\phi^*)^{p+1}(1-s^{q-1})\frac{(5-q)}{q+1}\left(\frac{a_p(5-p)(q+1)}{a_q(5-q)(p+1)}\frac{(1-s^{p-1})}{(1-s^{q-1})}+(\phi^*)^{q-p}\right)
  \\
  =2a_p (\phi^*)^{p+1}(1-s^{q-1})\frac{(5-q)}{q+1}\left(\frac{(5-p)(q+1)}{(5-q)(p+1)}\frac{(1-s^{p-1})}{(1-s^{q-1})}-\frac{p-1}{q-1}{\frac {q+1}{p+1}}\right)
  \\
    =2a_p (\phi^*)^{p+1}(1-s^{q-1})\frac{(5-q)}{p+1}\left(\frac{(5-p)}{(5-q)}\frac{(1-s^{p-1})}{(1-s^{q-1})}-\frac{p-1}{q-1}\right)>0,
\end{multline*}
where we have used in particular Lemma \ref{lem:f_bijective} for the last inequality. This implies that $F(\phi^*)>0$ which, since $J(\omega,p,q)=C(\phi_0)F(\phi_0)$, finishes the proof.
\end{proof}


\section{Determination of the sign of the slope}
\label{sec:sign}

In this section, we determine for each possible values of $a_p$, $a_q$, $p$ and $q$ the sign of $J(\omega,p,q)$. Combined with the stability criteria of Section \ref{sec:stab-crit}, this will prove Theorem \ref{thm:main}.
The general strategy of our proofs is the following.
Recall from Lemma \ref{lem:J=CF} that $J(\omega,p,q)=C(\phi_0)F(\phi_0),$ where 
\[
F(\phi_0)=\int_0^1 \frac{(5-p)\Phi_p+(5-q)\Phi_q}{(\Phi_p+\Phi_q)^{\frac{3}{2}}}s ds,
\]
and $C(\phi_0)>0$. Moreover, $\omega$ and $\phi_0$ are in an increasing one to one correspondence.
Hence, to determine the sign of $J$, it is sufficient to determine the sign of $F(\phi_0)$. To do this, we have two ingredients at our disposal. First, it is usually not difficult to establish that $F$ has a constant sign on intervals of the type $(\phi_*,\phi_{0,1})$ or $(\phi_{0,1}, \phi^*)$. On the other hand, the expression for $\partial_{\phi_0}F(\phi_0)$ given in Lemma \ref{lem:derivative_of_I} allows us to show that $\partial_{\phi_0}F(\phi_0)$ has a constant sign on intervals of the type $(\phi_*,\phi_{0,2})$ or $(\phi_{0,2}, \phi^*)$. If the intervals of the two ingredients overlap and if the signs are matching, the conclusion will follow. For example, if $F(\phi_0)>0$ on $(\phi_*,\phi_{0,1})$, and $\partial_{\phi_0}F(\phi_0)>0$ on $(\phi_{0,2}, \phi^*)$ and $\phi_{0,1}>\phi_{0,2}$, then $F(\phi_0)>0$ on $(\phi_*, \phi^*)$.  The detail of each case is given in the following sections.

\subsection{The focusing-focusing case}

In this section, we consider the case $a_p>0$, $a_q > 0$. In this case we have $\Phi_p>0$ and $\Phi_q>0$.

\begin{lemma}
  Let $a_p>0$, $a_q > 0$ and $q\leq 5$. Then for all $\omega \in(0,\infty)$ we have
  \[
    J(\omega,p,q)>0,
  \]
  and the family of standing waves is of type S. 
\end{lemma}

\begin{proof}
  If $q \leq 5 $, then $5-p>0$ and $ 5-q\geq 0$. Therefore for any $\phi_0\in(0,\infty)$  we have
  \[
    F(\phi_0)>0,
  \]
  which gives the desired conclusion.
\end{proof}

\begin{lemma}
Let $a_p>0$, $a_q > 0$ and $p \geq 5$. Then for all $\omega \in(0,\infty)$ we have
  \[
    J(\omega,p,q)<0,
  \]
  and the family of standing waves is of type U. 
\end{lemma}

\begin{proof}
  If $p \geq 5$, then $5-p \leq 0$ and $5-q<0$.  Therefore for any $\phi_0\in(0,\infty)$  we have
  \[
    F(\phi_0)<0,
\]
  which gives the desired conclusion.
\end{proof}

The remaining case $p<5<q$ is a bit more involved to consider.

\begin{lemma} \label{lem:FF:J<0}
Let $a_p>0$, $a_q > 0$ and $p<5<q$. There exists $\phi_{0,1}$ (explicitly given in \eqref{eq:FF:phi_0,1}) such that if $\phi_0^{q-p}>\phi_{0,1}^{q-p}$ then
\[
F(\phi_0)<0.
\]
\end{lemma}

\begin{proof} 
Using the formula \eqref{eq:F(phi_0)} of $F(\phi_0)$ and replacing in the numerator of the integrand $\Phi_p$ and $\Phi_q$ by their expressions \eqref{eq:phi_p,phi_q}, we obtain 
\[
F(\phi_0)= \int_0^1 \frac{(5-p)\frac{2a_p}{p+1}\phi_0^{p+1}(s-s^{p})+(5-q)\frac{2a_q}{q+1}\phi_0^{q+1}(s-s^{q})}{(\Phi_p+\Phi_q)^{\frac{3}{2}}}ds.
\]
Let 
\[
l(s)=(5-p)\frac{2a_p}{p+1}\phi_0^{p+1}(s-s^{p})+(5-q)\frac{2a_q}{q+1}\phi_0^{q+1}(s-s^{q}),
\]
and 
\[
k(s)=(\Phi_p+\Phi_q)^{\frac{3}{2}}.
\]
We may reformulate $l(s)$ in the following way:
\[
l(s)=\left((5-p)\frac{2a_p}{p+1}\phi_0^{p+1}+(5-q)\frac{2a_q}{q+1}\phi_0^{q+1}\frac{1-s^{q-1}}{1-s^{p-1}}\right) (s-s^{p}).
\]
From Lemma \ref{lem:f_bijective}, we know that the function $s\to \frac{1-s^{q-1}}{1-s^{p-1}}$ is increasing from $1$ to $\frac{q-1}{p-1}$ when $s$ goes from $0$ to $1$. Let  $\phi_{0,0}$ be given by
\[
\phi_{0,0}^{q-p}=-\frac{a_p(5-p)(q+1)(p-1)}{a_q(5-q)(p+1)(q-1)},
\]
and assume from now on that
$\phi_0>\phi_{0,0}$. Then
\[
\lim_{s\to 1}\frac{l(s)}{s-s^p}<0,
\]
and there exists $s^*\in[0,1)$ such that 
$l(s)>0$ for $s \in (0,s^*)$ and $l(s)<0$ for $s \in (s^*,1)$. 

Define $\tilde k$ by $\tilde k(s)=\frac{k(s)}{k (s^*)}$. Then $\tilde k(s^*)=1$.
As $k$ and therefore $\tilde k$ is a  positive decreasing function of $s$, for all $s \in (0,1)$ we have
\[
\frac{l(s)}{\tilde k(s)}<l(s).
\]
Integrating over $(0,1)$, we obtain
\[
F(\phi_0)< \frac{1}{k (s^*)}\int_0^1l(s)ds,
\]
and $F(\phi_0)$ will be negative if the integral in the right member is. 
Define $\phi_{0,1}>\phi_{0,0}$ by
\begin{equation} \label{eq:FF:phi_0,1}
\phi_{0,1}^{q-p}=-\frac{a_p(5-p)(p-1)(q+1)^2}{a_q(5-q)(q-1)(p+1)^2}.
\end{equation}
If $\phi_0>\phi_{0,1}$, then
\[
\int_0^1 l(s)ds=(5-p)\frac{a_p}{p+1}\phi_0^{p+1}\left( \frac{p-1}{p+1}\right)+(5-q)\frac{a_q}{q+1}\phi_0^{q+1}\left( \frac{q-1}{q+1}\right)
<0.
\]
Hence for any $\phi_0>\phi_{0,1}$ we have $F(\phi_0)<0$. This concludes the proof.
\end{proof}

\begin{lemma} \label{lem:FF:gamma=(p+1)/2}
  Let $a_p>0$, $a_q > 0$ and $p<5<q$.
  Let $\gamma=\frac{p+1}{2}$.
There exists $\phi_{0,2}$ (explicitly given in \eqref{eq:FF:phi_0,2}) such that the integrand $I_\gamma$ of $F_\gamma$  defined in \eqref{eq:def-I-gamma} verifies
\[
  \frac{\partial I_{\gamma}}{\partial \phi_0} <0
\]
for all $\phi_0 \in (0, \phi_{0,2}).$
\end{lemma}

\begin{proof}
As $\gamma=\frac{p+1}{2}$, from Lemma \ref{lem:derivative_of_I} we have
\begin{align*}
\frac{\partial I_{\gamma}}{\partial \phi_0}&=\frac12\phi_0^{\gamma-1}
\left( \frac{\left((5-q)(p-q)\Phi_q \right)(\Phi_p+\Phi_q)-3(p-q)^2\Phi_p\Phi_q
      }{\left(   \Phi_p+\Phi_q
        \right)^{\frac52}}\right ),\\
&=\frac12\phi_0^{\gamma-1} \Phi_q (p-q)
\left( \frac{ (5+2q-3p)\Phi_p+(5-q)\Phi_q}{\left(   \Phi_p+\Phi_q
        \right)^{\frac52}}\right ).\\       
\end{align*}
As a consequence $\partial_{\phi_0}I_{\gamma}<0$ if 
\[
(5+2q-3p)\Phi_p+(5-q)\Phi_q>0.
\]
Replacing $\Phi_p$ and $\Phi_q$ by their expressions \eqref{eq:phi_p,phi_q}, this is equivalent to
\begin{equation} \label{eq:FF:5+2q-3p<0}
(5+2q-3p)\frac{a_p}{p+1}\phi_0^{p+1}(1-s^{p-1}) +(5-q)\frac{a_q}{q+1}\phi_0^{q+1}(1-s^{q-1})>0.
\end{equation}
Since $p<5<q$, we have $5+2q-3p>0$, and therefore \eqref{eq:FF:5+2q-3p<0} becomes 
\begin{equation} \label{eq:z}
\phi_0^{q-p}<-\frac{a_p}{a_q}\frac{(5+2q-3p)}{(5-q)}\frac{(q+1)}{(p+1)}\frac{(1-s^{p-1})}{(1-s^{q-1})}.
\end{equation}
We know from Lemma \ref{lem:f_bijective} that 
\[
\frac{p-1}{q-1}<\frac{1-s^{p-1}}{1-s^{q-1}}.
\]
Define 
\begin{equation} \label{eq:FF:phi_0,2}
\phi_{0,2}^{q-p}=-\frac{a_p}{a_q}\frac{(5+2q-3p)}{(5-q)}\frac{(q+1)}{(p+1)}\frac{(p-1)}{(q-1)}.
\end{equation}
If $\phi_0<\phi_{0,2}$
then \eqref{eq:z} is verified, which
concludes the proof.
\end{proof}

\begin{lemma}
  \label{lem:FF:F_gamma_has_at_most_one_zero}
Let $a_p>0$, $a_q > 0$ and $p<5<q$.
The function $F_{\gamma}(\phi_0)$ has at most one zero in $(0,\infty)$.
\end{lemma}

\begin{proof}
As $p<q$, we have
\[
3(p-1)(p-q)<0,
\]
hence
\[
(5-p)(q+1) < (p+1)(5+2q-3p).
\]
It implies
\[
-\frac{a_p}{a_q}\frac{(5-p)}{(5-q)} \frac {(q+1)^2}{(p+1)^2}\frac{(p-1)}{(q-1)} < -\frac{a_p}{a_q} \frac{(5+2q-3p)}{(5-q)} \frac{(q+1)}{(p+1)}\frac{(p-1)}{(q-1)}.
\]
Therefore, we have 
\[
\phi_{0,1}^{q-p} < \phi_{0,2}^{q-p}.
\]
We know from Lemma \ref{lem:FF:J<0} that $F_\gamma(\phi_0)<0$ if $\phi_0 \in (\phi_{0,1},\infty)$, and from Lemma \ref{lem:FF:gamma=(p+1)/2} that $F_\gamma(\phi_0)$ is decreasing for all $\phi_0 \in (0,\phi_{0,2})$. As $\phi_{0,1}< \phi_{0,2}$, this implies that $F_\gamma(\phi_0)$ has at most one zero.
\end{proof}

\begin{lemma}
  \label{lem:FFfinal}
  Let $a_p>0$, $a_q > 0$ and $p<5<q$. There exists $\omega_1\in(0,\infty)$ such that
  \[
J(\omega,p,q)>0\text{ for }\omega<\omega_1,\quad
J(\omega_1,p,q)=0,\quad
J(\omega,p,q)<0\text{ for }\omega>\omega_1,
    \]
  and the family of standing waves is of type SU. 
\end{lemma}

\begin{proof}
  From Proposition \ref{prop:sign_J_0_a_p_positive}, we know that $J(\omega,p,q)>0$ for $\omega$ close to $0$. Combined with Lemmas \ref {lem:FF:J<0} and \ref{lem:FF:F_gamma_has_at_most_one_zero}, this implies the desired result.
\end{proof}

\subsection{The focusing-defocusing case}

In this section, we consider the case $a_p>0$, $a_q < 0$. In this case $\Phi_p>0$ and $\Phi_q<0$.

\begin{lemma}
  \label{lem:FD_sub_sup}
  Let $a_p>0$, $a_q<0$ and $p\leq 5<q $. For any $\omega\in(0,\omega^*)$, we have
  \[
    J(\omega,p,q)>0,
  \]
  and the family of standing waves is of type $S$.
\end{lemma}

\begin{proof}
We have $5-p\geq 0$, and $5-q<0$. Therefore $F(\phi_0)>0$ for any $\phi_0\in(0,\phi^*)$, which gives the desired result.
\end{proof}

\begin{lemma} \label{lem:FD:gamma=(p+1)/2}
  Let  $a_p>0$, $a_q<0$ and $p<q\leq 5$. Let $\gamma=\frac{p+1}{2}$.
Then the integrand $I_\gamma$ of $F_\gamma$  defined in \eqref{eq:def-I-gamma} verifies
\[
  \frac{\partial I_{\gamma}}{\partial \phi_0} >0
\]
for all $\phi_0 \in (0, \phi^*)$.
\end{lemma}

\begin{proof}
From Lemma \ref{lem:derivative_of_I} with $\gamma=\frac{p+1}{2}$, we have
\[
\frac{\partial I_{\gamma}}{\partial \phi_0}=\frac12\phi_0^{\gamma-1}
\left( \frac{\left((5-q)(p-q)\Phi_q \right)(\Phi_p+\Phi_q)-3(p-q)^2\Phi_p\Phi_q
      }{\left(   \Phi_p+\Phi_q
        \right)^{\frac52}}\right ).
        \]
Since $ 5-q\geq 0$,  $p-q<0$ and $\Phi_q<0$, we have $\frac{\partial I_{\gamma}}{\partial \phi_0}>0$ for any $\phi_0 \in (0, \phi^*)$.
\end{proof}

\begin{lemma} \label{lem:FD:gamma=p-q+3}
  Let  $a_p>0$, $a_q<0$ and $5<p<q$. Let $\gamma=p-q+3$.
Then the integrand $I_\gamma$ of $F_\gamma$  defined in \eqref{eq:def-I-gamma} verifies
\[
  \frac{\partial I_{\gamma}}{\partial \phi_0} >0
\]
for all $\phi_0 \in (0, \phi^*)$.
\end{lemma}

\begin{proof}
Let $\gamma=p-q+3$.
\[
\frac{\partial I_{\gamma}}{\partial \phi_0}=\frac12\phi_0^{\gamma-1}
\left( \frac{\left((5-p)(p-2q+5)\Phi_p+(5-q)(2p-3q+5)\Phi_q \right)(\Phi_p+\Phi_q)-3(p-q)^2\Phi_p\Phi_q
      }{\left(   \Phi_p+\Phi_q
        \right)^{\frac52}}\right ).
  \]
  The sign of $\frac{\partial I_{\gamma}}{\partial \phi_0}$ is the same as the sign of the numerator of the fraction. Factoring out $\Phi_p^2$, the sign is the same as the one of the
second order polynomial in $\frac{\Phi_q}{\Phi_p}$ given by
\[
    (5-q)(2p-3q+5) \left( \frac{\Phi_q}{\Phi_p}\right)^2 +2(5-p)(2p-3q+5) \left( \frac{\Phi_q}{\Phi_p}\right)+(5-p)(p-2q+5).
\]
As $2p-3q+5<0$ and $5-q<0$,  the coefficient of the term of order $2$ is positive. Therefore to show that the polynomial is positive, it is sufficient to show that the discriminant $\Delta$, given by
\begin{align*}
\Delta&=4(5-p)(2p-3q+5)((5-p)(2p-3q+5)-(5-q)(p-2q+5)),\\
&=-8(5-p)(2p-3q+5)(p-q)^2,
\end{align*}
is negative.
We have $2p-3q+5<0$ and $5-p<0$, therefore $\Delta<0$.
This concludes the proof.
\end{proof}

\begin{lemma}
  \label{lem:FDfinal}
  Let $a_p>0$, $a_q < 0$.
  \begin{itemize}
  \item Let $p<q\leq 5$. Then for any $\omega\in(0,\omega^*)$, we have
  \[
    J(\omega,p,q)>0,
  \]
  and the family of standing waves is of type $S$.
    \item Let $5<p<q$. Then 
    there exist $\omega_1\in(0,\infty)$ such that
    \[
      J(\omega,p,q)<0\text{ for }\omega<\omega_1,\quad J(\omega_1,p,q)=0,\quad J(\omega,p,q)>0\text{ for }\omega>\omega_1,
    \]
    and the family of standing waves is of type US.
  \end{itemize}
\end{lemma}

\begin{proof}
  In both cases, we infer from Lemmas \ref{lem:FD:gamma=(p+1)/2} and \ref{lem:FD:gamma=p-q+3} that for any $\omega\in(0,\omega^*)$, the function
  $\omega\to J(\omega,p,q)$ changes sign (from negative to positive) at most once on $\omega\in(0,\omega^*)$.

To establish the desired conclusion, we consider the values of $J$ close to the endpoints. As $\omega\to0$, we have established in Proposition \ref{prop:sign_J_0_a_p_positive} that for $\omega$ close to $0$, we have
   \[
      J(\omega,p,q)>0\text{ for }p\leq 5, \quad J(\omega,p,q)<0\text{ for }p>5.
    \]
    As $J$ is increasing, this gives the conclusion for the first part of the Lemma.

    For the second part of the Lemma, we look at the limit $\omega\to \omega^*$ (i.e.~$\phi_0\to\phi^*$). From Proposition \ref{prop:sign_J^*_a_q_negative}, for $5<p<q$ and for $\omega$ close to $\omega^*$ we have
     \[
      J(\omega,p,q)>0,
    \]
  which gives the second part of the Lemma.
\end{proof}

\subsection{The defocusing-focusing case}
In this section, we consider the case $a_p<0$, $a_q > 0$. In this case $\Phi_p<0$ and $\Phi_q>0$.

\begin{lemma} \label{lem:DF:gamma=(q+1)/2}
Let  $a_p<0$, $a_q >0$ and $p<q<5$. Let $\gamma=\frac{q+1}{2}$.
If $3q \geq 2p+5$, then the integrand $I_\gamma$ of $F_\gamma$  defined in \eqref{eq:def-I-gamma} verifies
\[
  \frac{\partial I_{\gamma}}{\partial \phi_0} >0
\]
for all $\phi_0 \in (\phi_*, \infty)$.
\end{lemma}

\begin{proof}
As $\gamma=\frac{q+1}{2}$, from Lemma \ref{lem:derivative_of_I} we have
\begin{align*}
\frac{\partial I_{\gamma}}{\partial \phi_0}&=\frac12\phi_0^{\gamma-1}
 \left( \frac{\left((5-p)(q-p)\Phi_p \right)(\Phi_p+\Phi_q)-3(p-q)^2\Phi_p\Phi_q
       }{\left(   \Phi_p+\Phi_q
         \right)^{\frac52}}\right ),\\
&=\frac12\phi_0^{\gamma-1} \Phi_p (q-p)
 \left( \frac{ (5-p)\Phi_p+(5-p)\Phi_q+3(p-q)\Phi_q}{\left(   \Phi_p+\Phi_q
         \right)^{\frac52}}\right ),\\
&=\frac12\phi_0^{\gamma-1} \Phi_p (q-p)
\left( \frac{ (5-p)\Phi_p+(5+2p-3q)\Phi_q}{\left(   \Phi_p+\Phi_q
         \right)^{\frac52}}\right ).\\       
\end{align*}
As $0<5-p$ and $2p+5-3q \leq 0$ we have
\begin{align*}
(5-p)\Phi_p+(5+2p-3q)\Phi_q &<0.
\end{align*}
As a consequence $\frac{\partial I_{\gamma}}{\phi_0}>0$ for all $\phi_0 \in (\phi_*, \infty)$ when $5+2p-3q\leq 0$, which is the desired conclusion.
\end{proof}

\begin{lemma} \label{lem:DF:gamma=p-q+3}
  Let  $a_p<0$, $a_q > 0$ and  $p<q<5$. Let $\gamma=p-q+3$. If $3q<2p+5$ then the integrand $I_\gamma$ of $F_\gamma$  defined in \eqref{eq:def-I-gamma} verifies
  \[
    \frac{\partial I_{\gamma}}{\partial \phi_0} >0
  \]
  for all $\phi_0 \in (\phi_*, \infty).$
\end{lemma}

\begin{proof}
  As $\gamma=p-q+3$, from Lemma \ref{lem:derivative_of_I} we have

\[
\frac{\partial I_{\gamma}}{\partial \phi_0}=\frac12\phi_0^{\gamma-1}
 \left( \frac{\left((5-p)(p-2q+5)\Phi_p+(5-q)(2p-3q+5)\Phi_q \right)(\Phi_p+\Phi_q)-3(p-q)^2\Phi_p\Phi_q
       }{\left(   \Phi_p+\Phi_q
         \right)^{\frac52}}\right ).
\]
If the numerator of the fraction is positive then the derivative is positive. Factorizing out $\Phi_p^2$, the sign of the numerator is the same as the one of the quadratic polynomial in $\frac{\Phi_q}{\Phi_p}$ given by
\[
  (5-q)(2p-3q+5) \left( \frac{\Phi_q}{\Phi_p}\right)^2 +2(5-p)(2p-3q+5) \left( \frac{\Phi_q}{\Phi_p}\right)+(5-p)(p-2q+5).
\]
As $2p-3q+5>0$ and $5-q>0$, the coefficient of the term of order $2$ is positive. Therefore to show that the polynomial is positive, it is sufficient to show that the discriminant $\Delta$, given by
\begin{align*}
\Delta&=4(5-p)(2p-3q+5)((5-p)(2p-3q+5)-(5-q)(p-2q+5)),\\
&=-8(5-p)(2p-3q+5)(p-q)^2,
\end{align*}
is negative.
We have $2p-3q+5>0$ and $5-p>0$, therefore $\Delta<0$.
This concludes the proof.
\end{proof}

\begin{lemma}
  \label{lem:DFfinal}
  Let $a_p<0$, $a_q > 0$ and $p<q<5$.
  \begin{itemize}
  \item If $q\leq 7-2p$,  then for any $\omega\in(0,\infty)$, we have
  \[
    J(\omega,p,q)>0,
  \]
  and the family of standing waves is of type $S$.
    \item If $q> 7-2p$,  then
    there exist $\omega_1\in(0,\infty)$ such that
    \[
      J(\omega,p,q)<0\text{ for }\omega<\omega_1,\quad J(\omega_1,p,q)=0,\quad J(\omega,p,q)>0\text{ for }\omega>\omega_1,
    \]
    and the family of standing waves is of type US.
  \end{itemize}
\end{lemma}

\begin{proof}
  Lemmas \ref{lem:DF:gamma=(q+1)/2} and \ref{lem:DF:gamma=p-q+3} implies $F_\gamma$ changes sign only once on $(0,\omega^*)$. From Proposition \ref{propJ0}, we know that as $\omega\to0$, we have
    \(
    J(\omega,p,q)>0
  \)
when $q\leq 7-2p$, which gives the conclusion for the first part of the Lemma. When $q> 7-2p$, from Proposition \ref{propJ0}, we know that as $\omega\to0$, we have
    \(
    J(\omega,p,q)<0.
  \)
  Since, from Proposition \ref{prop:sign_J^*_a_q_positive}, we know that
  \(
    J(\omega,p,q)>0
  \)
  for large $\omega$, the conclusion follows for the second part of the Lemma.
\end{proof}

\begin{lemma}
  Let $a_p<0$, $a_q > 0$ and $p<5\leq q $.
  For any $\omega\in(0,\infty)$, we have
  \[
    J(\omega,p,q)<0,
  \]
  and the family of standing waves is of type $U$.
\end{lemma}

\begin{proof}
We have $ 5-p>0$, $5-q\leq 0$, $\Phi_p<0$ and $\Phi_q>0$. Therefore we directly see on the expression \eqref{eq:F(phi_0)} of $F(\phi_0)$ that $F(\phi_0)<0$, which gives the desired result.
\end{proof}

\begin{lemma}
  Let  $a_p<0$, $a_q > 0$ and $5\leq p<q$.
  For any $\omega\in(0,\infty)$, we have
  \[
    J(\omega,p,q)<0,
  \]
  and the family of standing waves is of type $U$.
\end{lemma}

\begin{proof}
We know that $\phi_*<\phi_0$, therefore $-\frac{a_p}{a_q}\frac{q+1}{p+1}<\phi_0^{q-p}$.
As $\frac{5-p}{5-q}<1$, we have $-\frac{a_p}{a_q}\frac{q+1}{p+1}\frac{5-p}{5-q}<\phi_0^{q-p} $.
From Lemma \ref{lem:f_bijective} we know that $\frac{1-s^{p-1}}{1-s^{q-1}}<1$, hence 
\[
-\frac{a_p}{a_q}\frac{(q+1)}{(p+1)}\frac{(5-p)}{(5-q)}\frac{1-s^{p-1}}{1-s^{q-1}}<\phi_0^{q-p},
\]
which is equivalent to
\[
\frac{(5-p)}{(5-q)}<-\frac{\Phi_q}{\Phi_p},
\]
which implies 
\[
(5-p)\Phi_p+(5-q)\Phi_q<0.
\]
This implies that $F(\phi_0)<0$ which gives the desired result.
\end{proof}

\subsection{The critical frequency}

Observe that, as a by-product of the analysis of the previous sections, we always have instability at the critical frequency when there is a stability change. Indeed, we have
\[
  \partial_\omega^2M(\phi_\omega)=\partial_\omega\left(C(\phi_0)\phi_0^{-\gamma}F_\gamma(\phi_0)\right)=
  \partial_\omega\phi_0\left(\partial_{\phi_0}\left(C(\phi_0)\phi_0^{-\gamma}\right) F_\gamma(\phi_0)+C(\phi_0)\phi_0^{-\gamma}\partial_{\phi_0}F_\gamma(\phi_0)\right).
  \]
  At the stability change, we have $F(\phi_0)=0$. Therefore, at the stability change,
  \[
\partial_\omega^2M(\phi_\omega)=(\partial_\omega\phi_0)C(\phi_0)\phi_0^{-\gamma}\partial_{\phi_0}F_\gamma(\phi_0).
\]
As we have shown that in this case $\partial_{\phi_0}F_\gamma(\phi_0)\neq 0$, the criterion \eqref{eq:comech-pelinovsky} holds.

\section{Numerical experiments}
\label{sec:numerics}
To explore further the stability/instability of standing waves, we have performed a series of numerical experiments in the case $a_p<0$, $a_q>0$, $1<p<q<5$.

The Python language and the specific libraries Numpy, Scipy and Matplotlib have been used to perform the experiments. The code is made available in \cite{KfLeTs-github}.

\subsection{The critical surface for stability/instability}

We first analyzed the critical surface in $(\omega,p,q)$ separating instability from stability. To this aim, we first have implemented the calculation of $J(\omega,p,q)$. The function \verb+integrate.quad+ has been used to perform the integration. While the results are overall satisfactory, in some cases the function returned incorrect results, with problems increasing as $\omega$ was taken closer to $0$.

To estimate the critical $\omega $ at given $(p,q)$, we have used the classical bisection method, which has the advantage of being very robust. The algorithm is divided into two parts.

First, we find an initial interval $[\omega_0,\omega_1]$ in which we are sure that $\omega\to J(\omega,p,q)$ changes sign. A natural choice for $\omega_0$ is $0$. To find a suitable $\omega_1$, we simply start with $\omega_1=1$ and test if $J(\omega_1,p,q)>0$. If not, we replace $\omega_1$ by $2\omega_1$ and repeat until $J(\omega_1,p,q)>0$. To avoid running an infinite loop, we break it when $\omega_1>10^{10}$ and do not search for $\omega_c$ in these cases. Second, we apply the bisection method to search for a root of $J(\omega,p,q)$ inside $[\omega_0,\omega_1]$. As this approach, while being efficient, is also relatively slow, we took advantage of the computer power of our department to run computations in parallel on the $(p,q)\in[1,5]\times [1,5]$ grid with $dp=dq=0.01$.

We have represented the critical surface
\[
  \{(p,q,\omega_c(p,q))\}
\]
for $a_p=-1$ and three different values of $a_q=1/2,1,2$ in Figure \ref{fig:critical_surface}. 
\begin{figure}[htpb!]
    \centering \includegraphics[width=.3\textwidth]{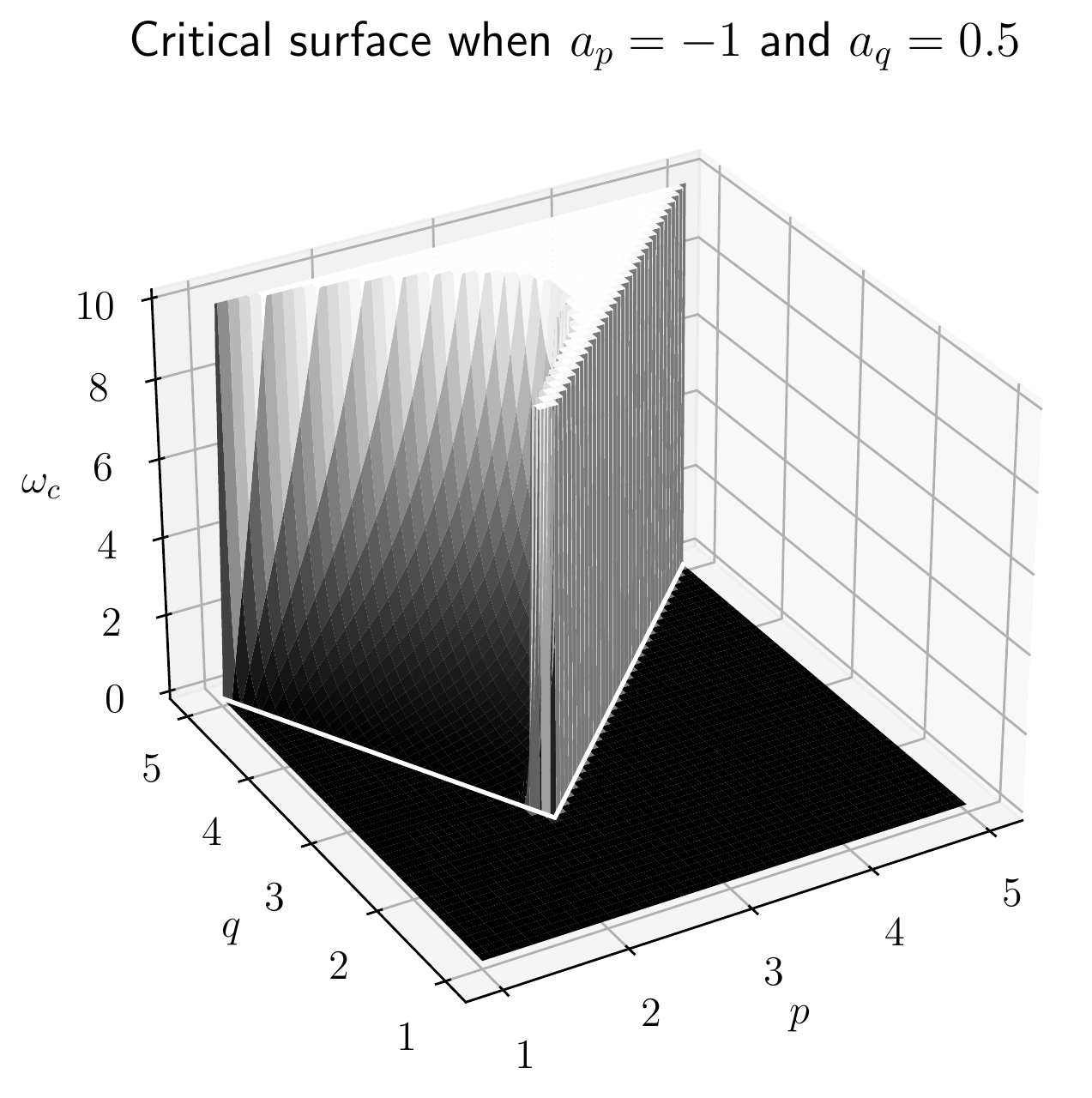}~\includegraphics[width=.3\textwidth]{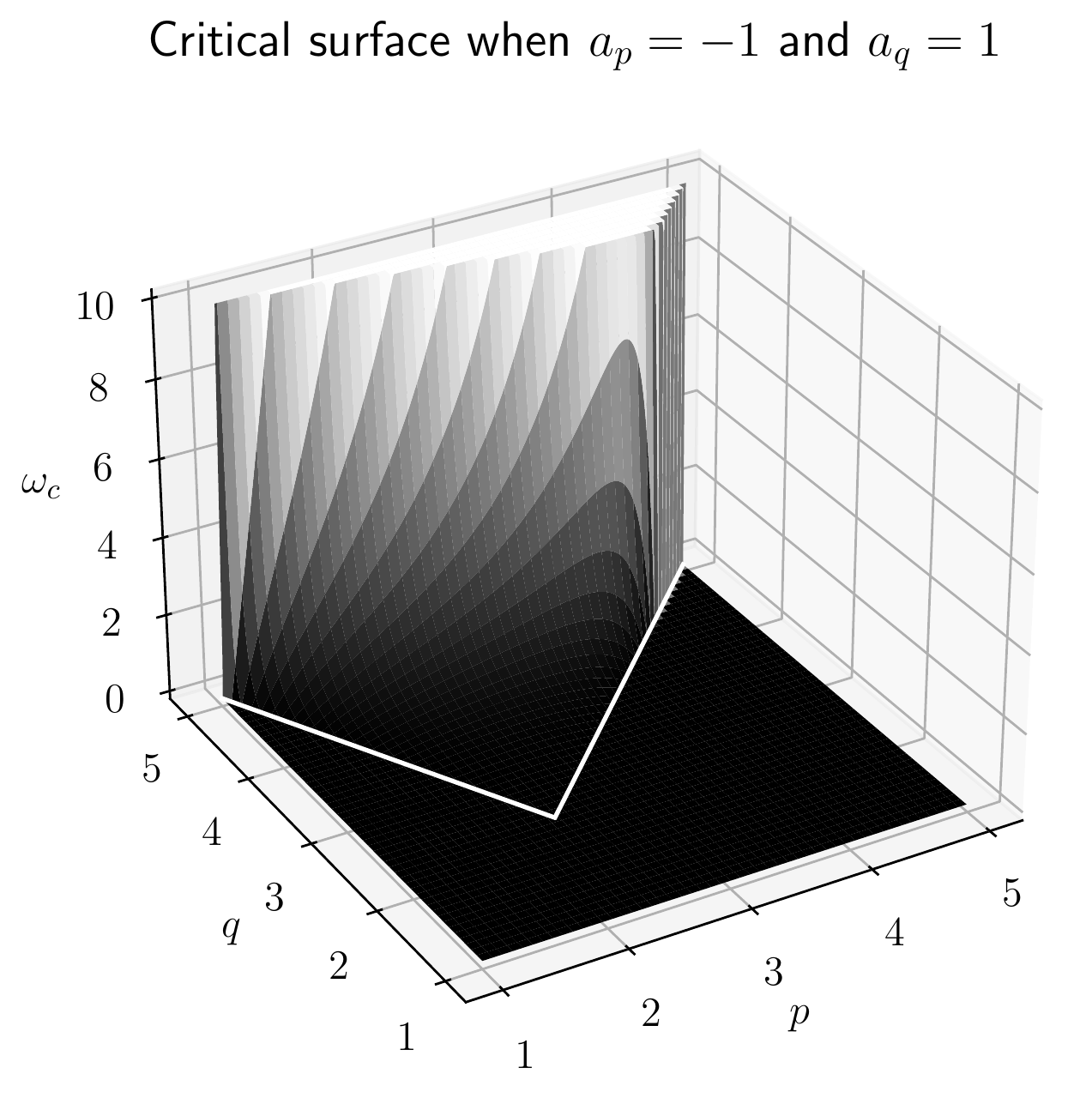}~\includegraphics[width=.3\textwidth]{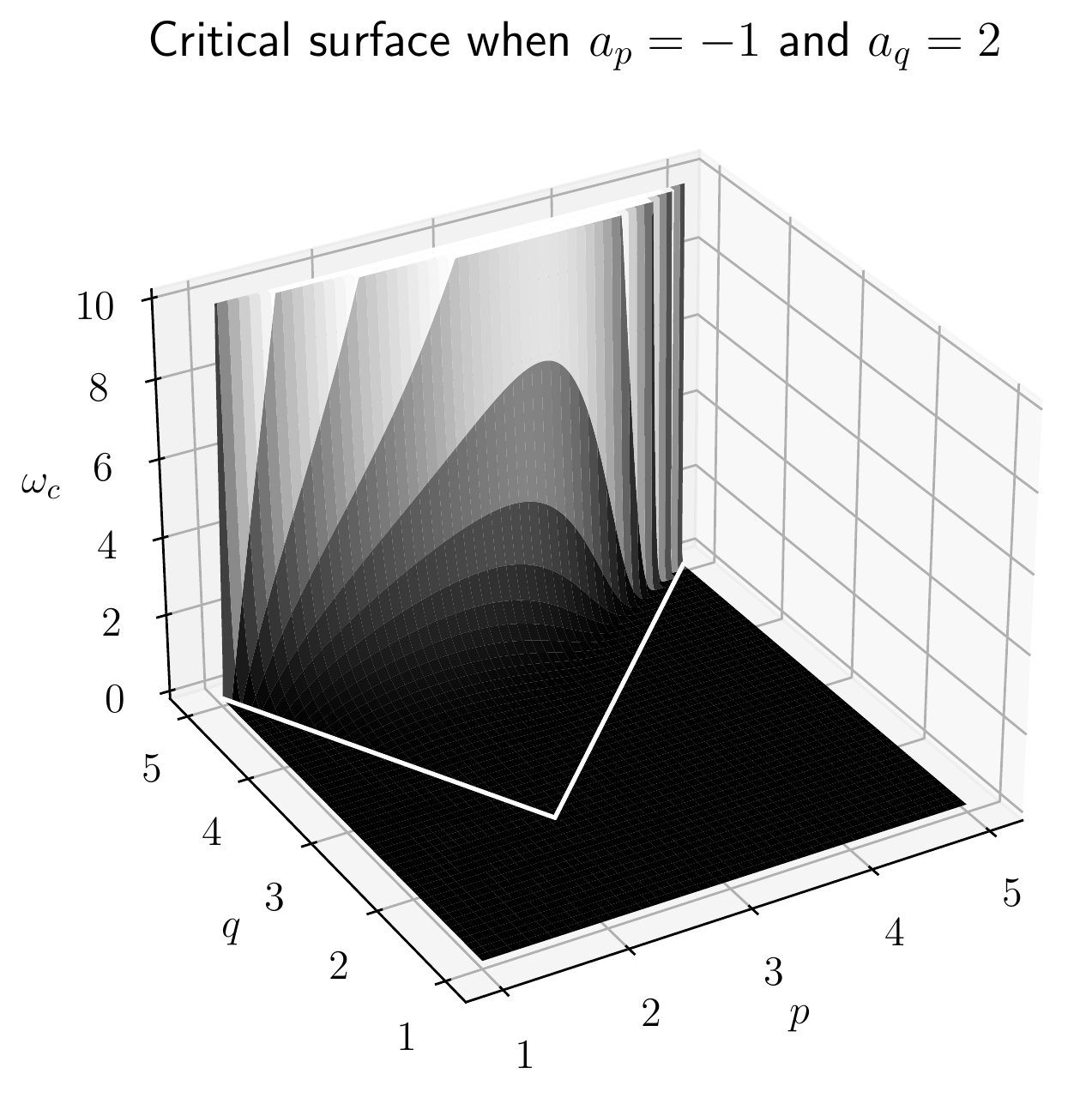}
  \caption[Frequency Critical Surface]{Critical surface $\{(p,q,\omega_c(p,q))\}$ for $a_p=-1$ and $a_q=1/2,1,2$. The white lines represent $q=7-2p$ and $q=p$, where the transition from $\omega_c=0$ to $\omega_c>0$ occurs.}
  \label{fig:critical_surface}
\end{figure}

Several observations can be made on the critical surface. As $(p,q)$ approaches the line $q=5$, we have $\omega_c(p,q)\to \infty$, which is consistent with the fact that standing waves are all unstable on this line.

It can be observed that on the line $q=7-2p$ the transition is continuous, no matter the value of $a_q$. To the contrary, the transition is continuous on the line $p=q$ when $a_q\geq 1$, whereas it becomes discontinuous when $a_q<1$, in which case $\omega_c(p,q)\to \infty$ as $q\to p$.

To investigate more the transition close to the lines $q=7-2p$ and $q=p$, we plot slices of the critical surface for a fixed value of $q$ in Figure \ref{fig:slices}. We chose to present the results when $q=4$, but similar results are obtained with other values of $q$. On Figure \ref{fig:slices}, we observe that when $a_q=|a_p|=1$, the transition between $\omega_c(p,q)=0$ and $\omega_c(p,q)>0$ at $q=p$ and $q=7-2p$ is Lipschitz. When $a_q=2>|a_p|=1$, the transition seems smoother (but closer observations will reveal otherwise) when $q=p$, whereas it remains Lipchitz when $q=7-2p$. To the contrary, when $a_q=1/2<|a_p|=1$,  the transition is discontinuous when $q=p$, whereas it seems smoother at $q=7-2p$.
\begin{figure}[htpb!]
  \centering
  \includegraphics[width=.3\textwidth]{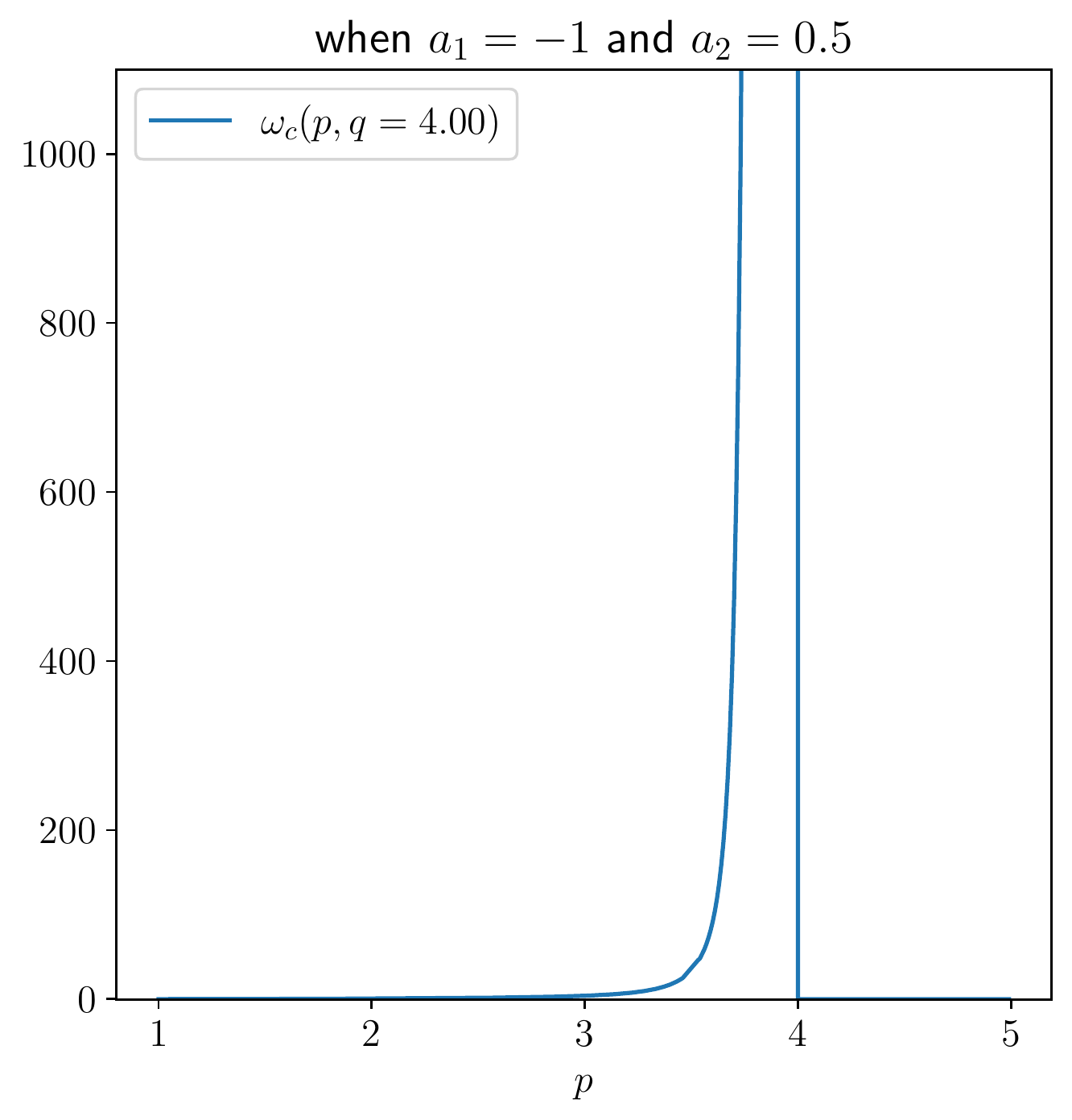}~
  \includegraphics[width=.3\textwidth]{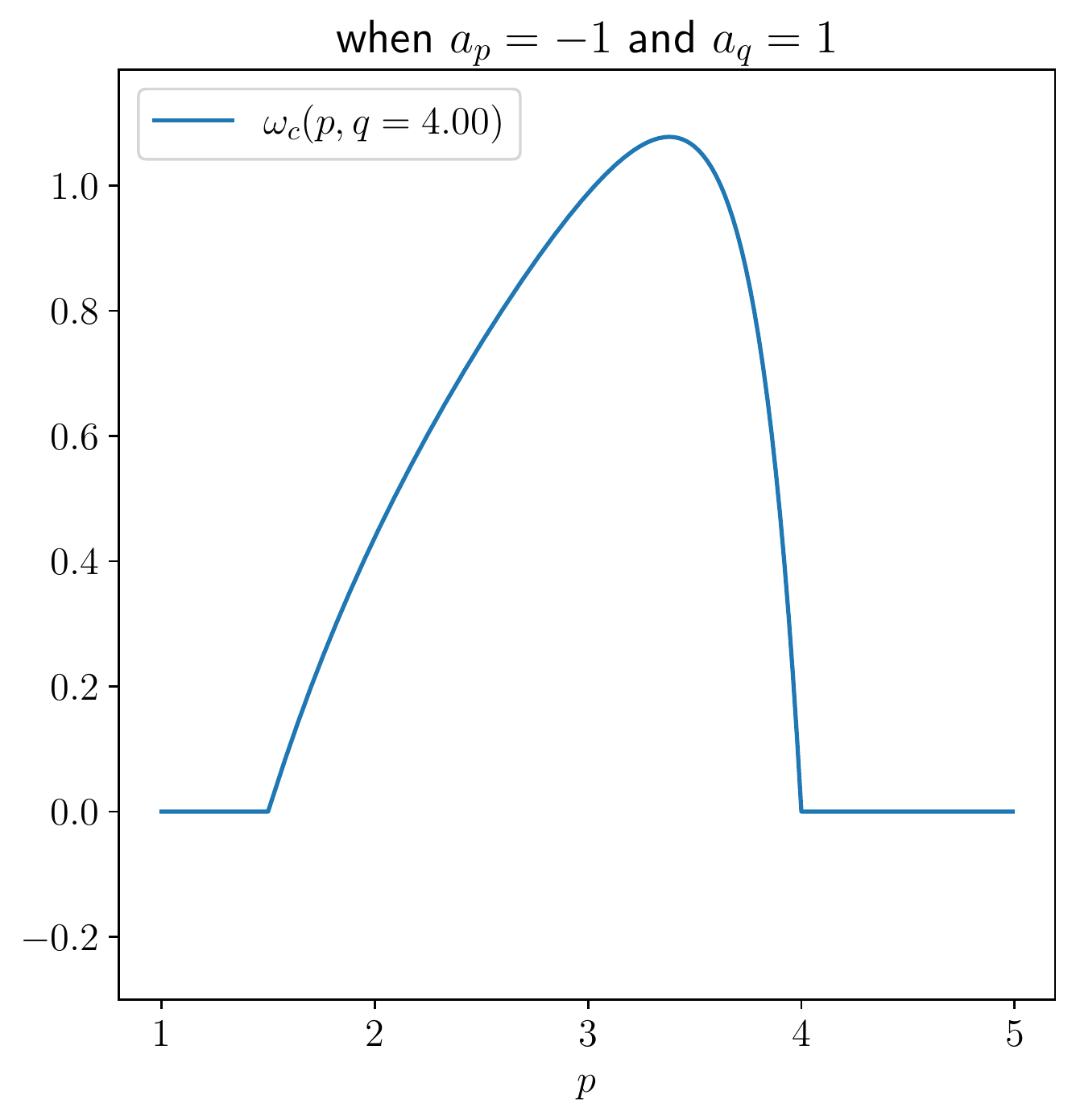}~
  \includegraphics[width=.3\textwidth]{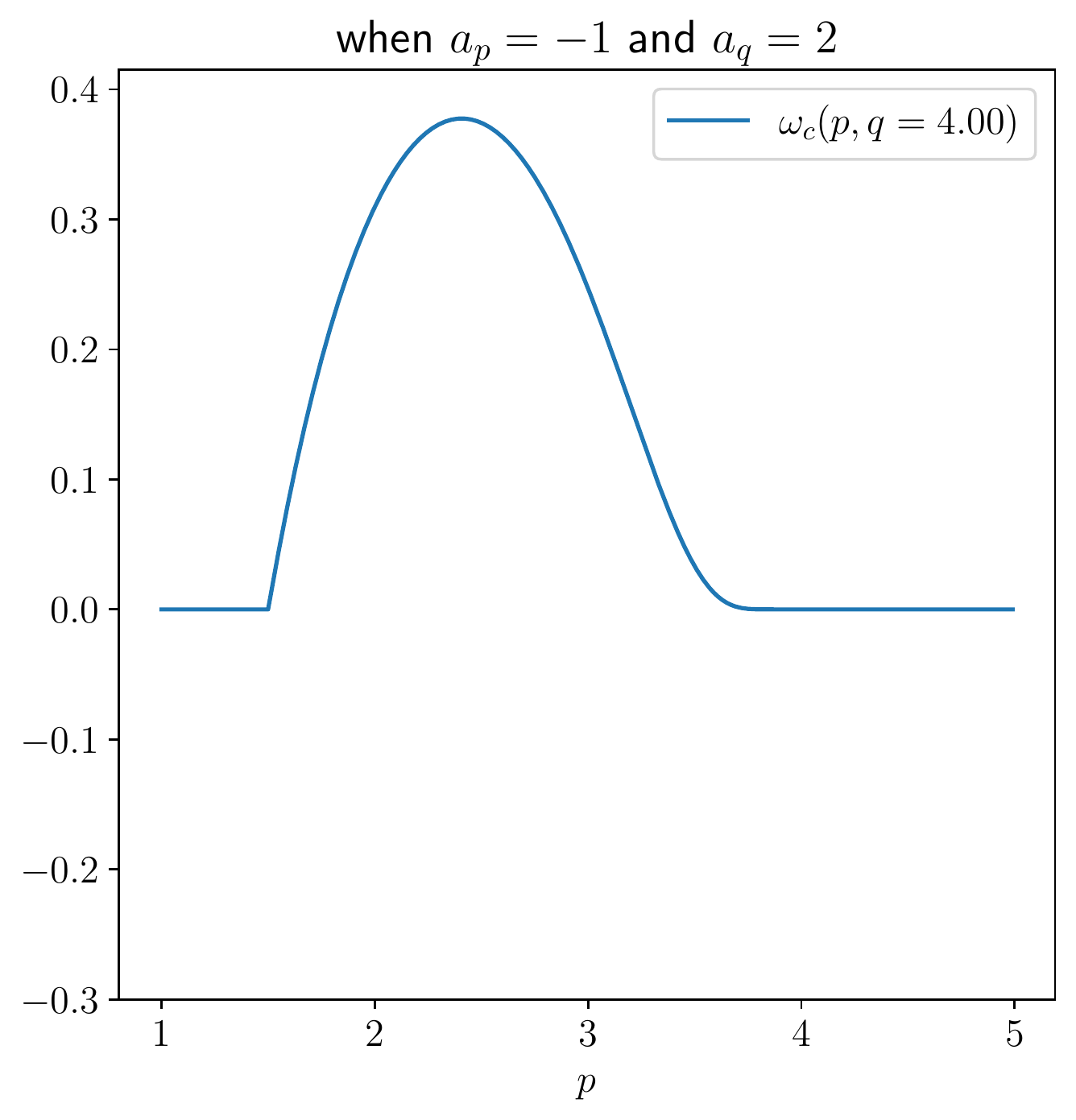}
  \caption[Slices of the critical surface]{Slices of the critical surface for fixed value of $q=4$}
  \label{fig:slices}
\end{figure}

To confirm our previous observations, we zoomed on the slices of Figure \ref{fig:slices} and obtained the results presented in Figure \ref{fig:slices_zoom}.
\begin{figure}[htpb!]
  \centering
  \includegraphics[width=.3\textwidth]{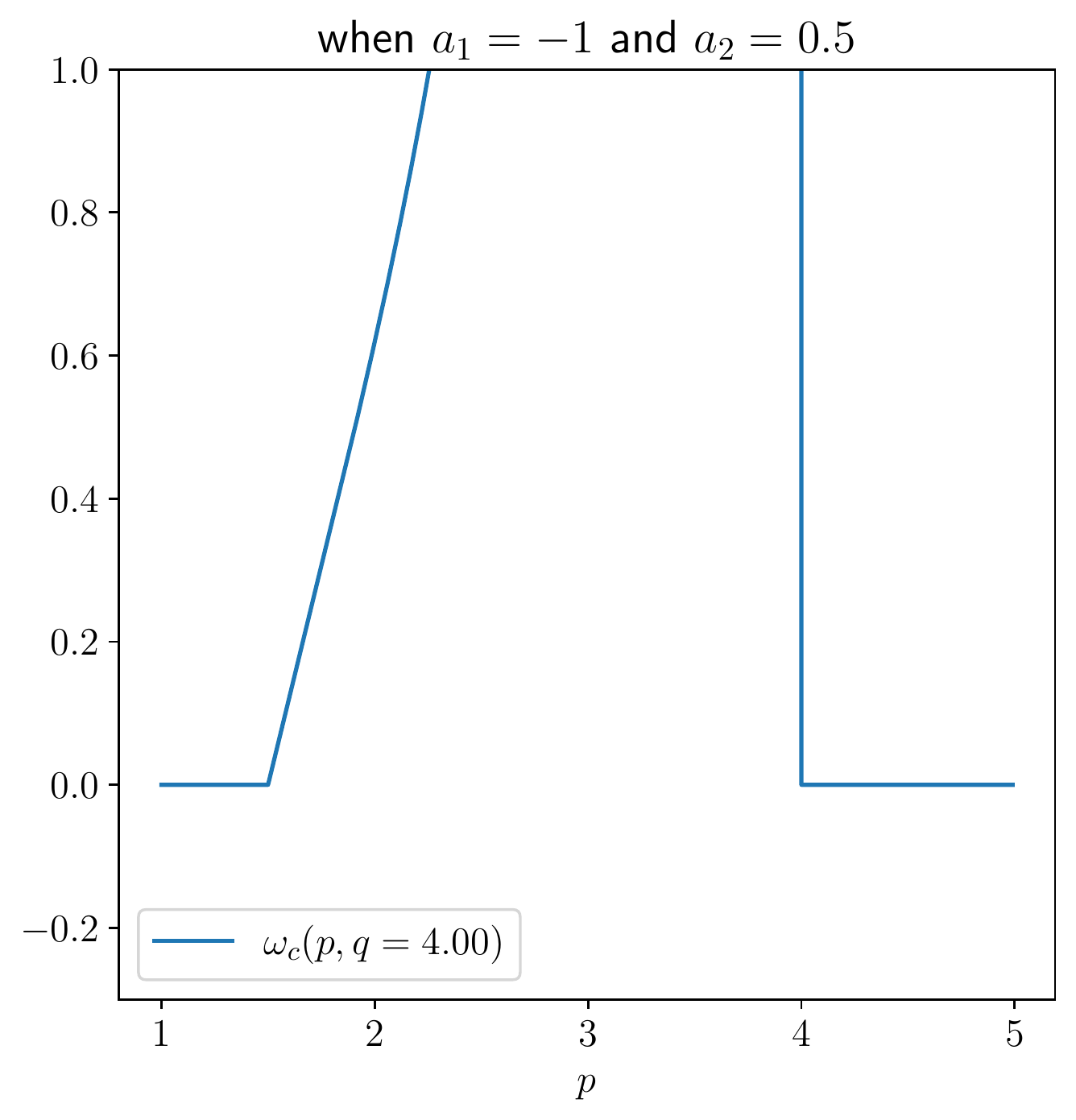}~
  \includegraphics[width=.3\textwidth]{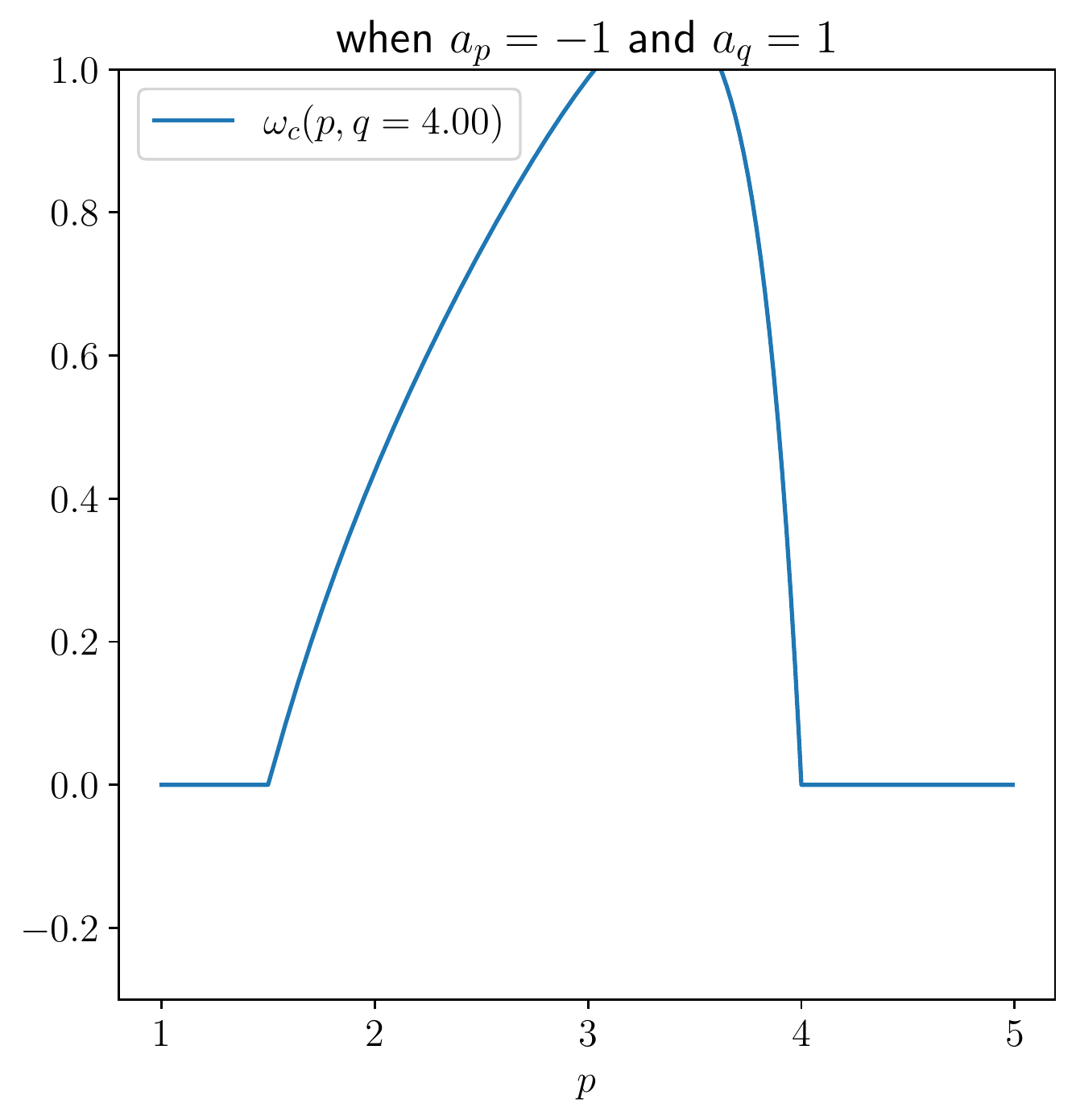}~
  \includegraphics[width=.3\textwidth]{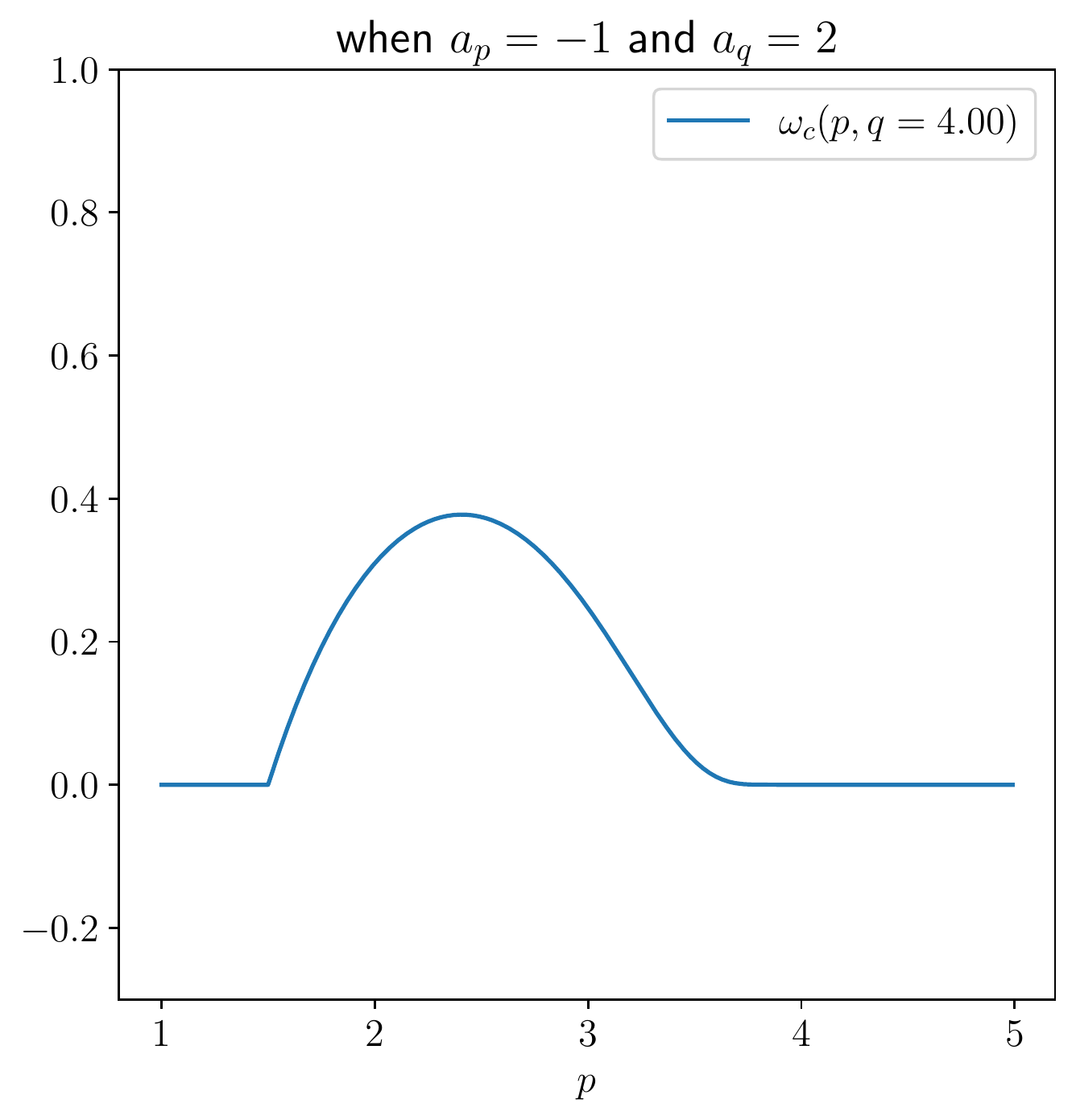}
  \caption[Slices of the critical surface]{Zoom on slices of the critical surface for fixed value of $q=4$}
  \label{fig:slices_zoom}
\end{figure}
Observing closer the transition from $\omega_c>0$ to $\omega_c=0$ on fixed $q$ slices of Figure \ref{fig:slices_zoom}, we realize that the transition on the left ($q=7-2p$) seems to be always only Lipschitz, contrary to what could be inferred from the previous observation. On the other hand, the previous observation when $p=q$ is confirmed:  the transition seems smooth when $a_q=2$, Lipschitz when $a_q=1$, and discontinuous when $a_q=1/2$. This is reflecting the fact that when $p\to q$, the family of soliton profiles has a different behavior for different values of $a_q$. When $a_q>|a_p|$, soliton profiles for $p=q$ exist and are stable (hence $\omega_c(p,q=p)=0$), whereas for $a_q=|a_p|$ the two nonlinearities exactly compensate and for $|a_p|>a_q$ the defocusing nonlinearity becomes the dominant one (and solitary waves do not even exist). 

From the previous observations, we know that at fixed $q$ the map $p\to \omega_c(p,q)$ has a unique maximum if $a_q=1$ or $a_q=2$ (if $a_q=1/2$, we have seen that the map increases towards infinity as $p$ approaches $q$). Denote by $p_{\max}(q)$ the value realizing this maximum, i.e.
\begin{equation*}
\omega_c(p_{\max}(q),q)=\max_{1<p<5}
\omega_c(p,q).
\end{equation*}
The line $\{(p_{\max}(q),q), q>7/3\}$ is represented in Picture \ref{fig:argmax}.
\begin{figure}[htpb!]
  \centering
  \includegraphics[width=.49\textwidth]{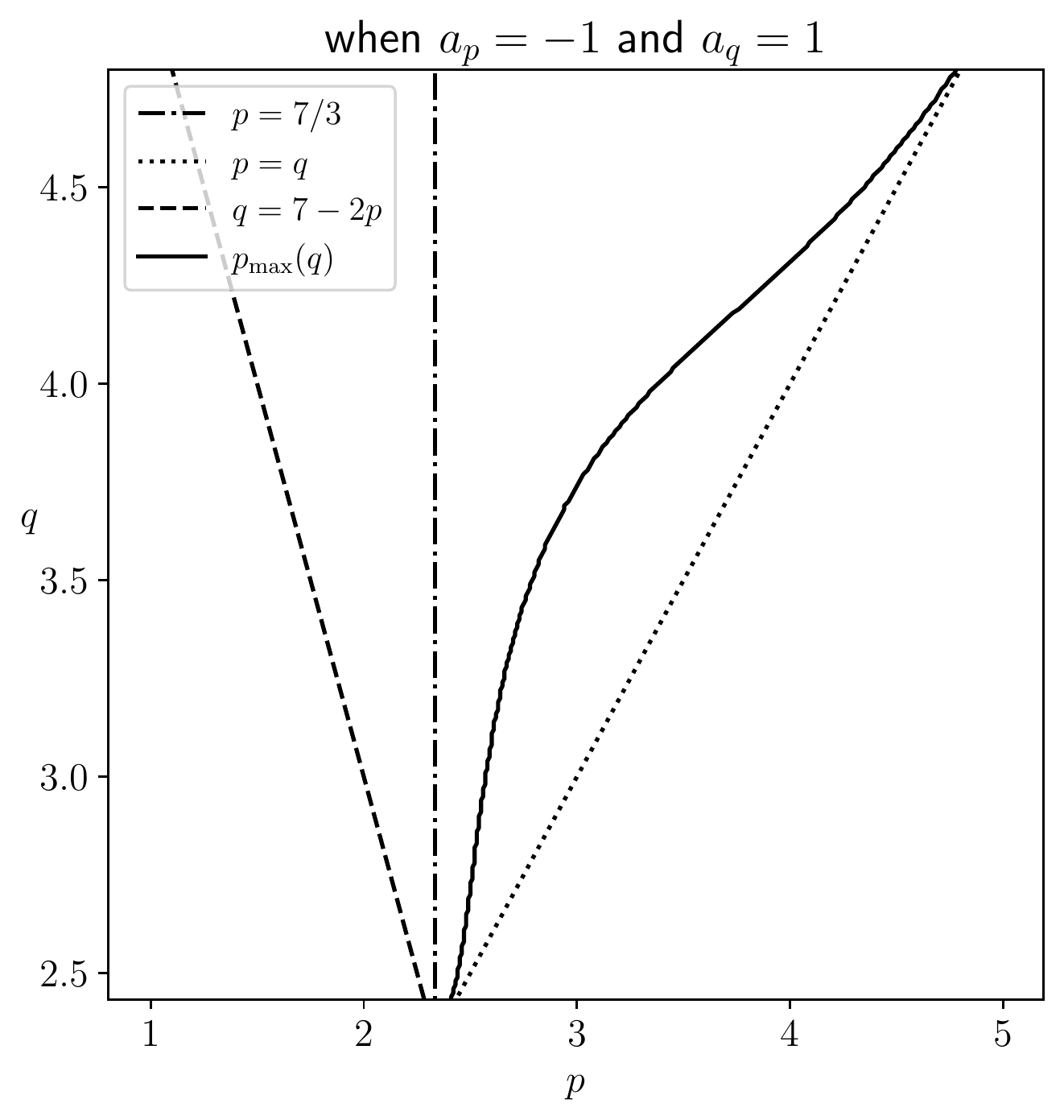}~
  \includegraphics[width=.49\textwidth]{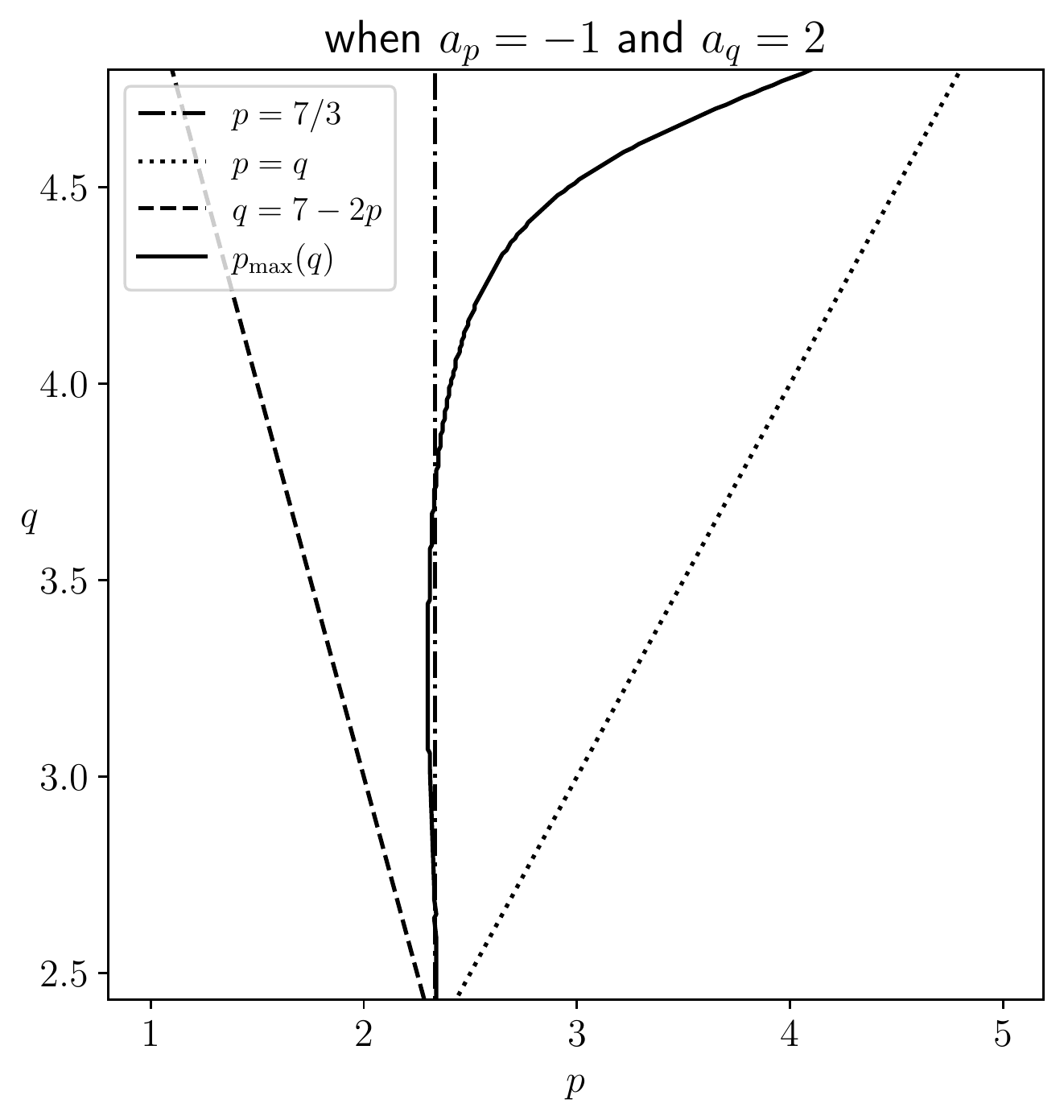}~
  \caption[Argmax curve]{Curve of the argument of $\max_p\omega_c(p,q)$ in terms of $q$}
  \label{fig:argmax}
\end{figure}
When $a_q=1$, we observe that the line is tangent to the line $p=q$ when $q$ is close to $7/3$ or close to $5$. On the other hand, when $a_q=2$, the line seems to be tangent to the line $p=7/3$ when $q$ is close to $7/3$. It approaches the point $(5,5)$ as $q$ goes to $5$, but does not seem to be tangent to the line $p=q$ (it was however not possible to obtain numerically a relevant picture closer to $q=5$, which leaves open the question of the behavior when $q$ is close to $5$).

\subsection{Evolution for initial data close to standing waves}

We now turn to numerical experiments for the stability/instability of solitary waves for the flow of \eqref{eq:nls}. For the experiments, we have used the Crank-Nicolson scheme with relaxation presented in \cite{Be04} which has been proved to be efficient for the numerical simulation of the Schr\"odinger flow (see e.g.~\cite{AnBaBe13} for the comparison of various schemes used for the dynamical simulations of the nonlinear Schr\"odinger flow).

For a time discretization step $\delta_t$ (typically $\delta_t=10^{-3}$), denote by $u^n$ the approximation of $u$ at time $t_n=n\delta_t$. The semi-discrete (in time) relaxation scheme is then given by
\[
  \begin{cases}
    \frac{\phi^{n+\frac12}+\phi^{n-\frac12}}{2}=a_p|u^n|^p+a_q|u^n|^q,\\
    i\frac{u^{n+1}-u^n}{\delta_t}+\partial_{xx}\left(\frac{u^{n+1}+u^n}{2}\right)=-\left(\frac{u^{n+1}+u^n}{2}\right)\phi^{n+\frac12},
  \end{cases}
\]
with the understanding that $u^0=u_0$ and $\phi^{-\frac12}=a_p|u^0|^p+a_q|u^0|^q$.
For the implementation, the scheme is further discretized in space with second order finite differences for the second derivative operator, with Dirichlet boundary conditions.

We have performed simulations for $(p,q)$ on the line $q=2p-1$, as for this range of exponents explicit formulas are available for solitary wave profiles (see e.g.~\cite{LiTsZw21}) and can be used easily to construct initial data. Considering other ranges of $(p,q)$ would have been possible, to the extend of additional computations to first obtain numerically solitary waves. As we do not expect different behavior to occur for other values of  $(p,q)$, the restriction to the line $q=2p-1$ is harmless.

The initial data that we construct are all based on a solitary wave profile $\phi_\omega$. They are of the form
\[
u_0=\phi_\omega+\eps\psi,
\]
where $0<\eps\ll 1$ is used to adjust the size of the perturbation and $\psi$ is the direction of perturbation, which can be for example
\[
\psi=\phi_\omega,\quad \psi=\phi_\omega \cos, \quad \psi=\phi_\omega  \tanh, \quad \psi= \phi_\omega (\cdot-3).
\]
As our numerical scheme uses Dirichlet conditions at the bounds of the space interval, we have chosen to work with well-localized perturbation in order to avoid possible numerical reflections due to the boundary conditions. 
Our experiments consisted in taking one of the previous possibility as initial data, running the simulation of the nonlinear Schr\"odinger flow, and observe the pattern of the outcome. It turns out that after running numerous simulations, we have observed only three possible types of behavior:
\begin{itemize}
\item Stability;
\item Growth followed by slightly decreasing oscillations;
  \item Dispersion.
  \end{itemize}
    Observe that our numerical results are in part similar to the ones obtained and discussed in further details  in \cite[Section 4]{CaKlSp20} in the case of the $2d$ cubic-quintic (focusing-defocusing) nonlinear Schr\"odinger equation.

  Stability means that the solution does not leave the neighborhood of $\phi_\omega$ (up to phase shift and translations). We obviously expect to see this behavior in the cases where the values of the parameters $p$, $q$, and $\omega$ ensure that the solitary wave will be stable. However, one thing which is not easily decided by the theory is the size of the basin of stability of the solitary wave. In other words, finding a perturbation of the solitary wave sufficiently large to be visible, but small enough so that the corresponding solution remains in the vicinity of the solitary wave requires delicate adjustments. 

  An example of a stable behavior is provided in Figure \ref{fig:stable}. Observe that while on the global scale the solution seems to be behave exactly as a solitary wave (left picture), when getting a closer look at the maximum value (right picture) we observe small oscillations (with an amplitude of order $0.03$).
 \begin{figure}[htpb!]
   \centering \includegraphics[width=.48\textwidth]{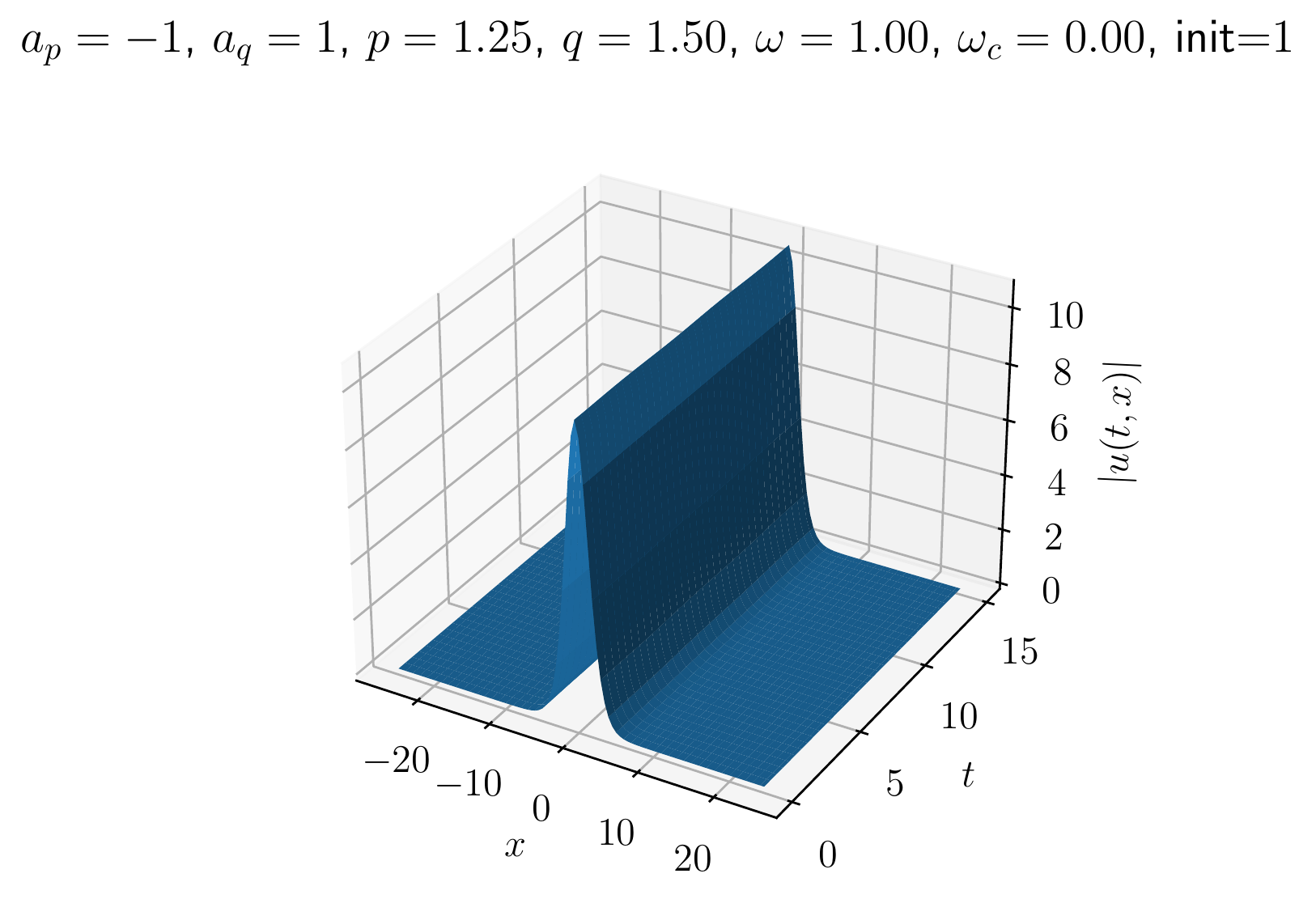}
   ~\includegraphics[width=.38\textwidth]{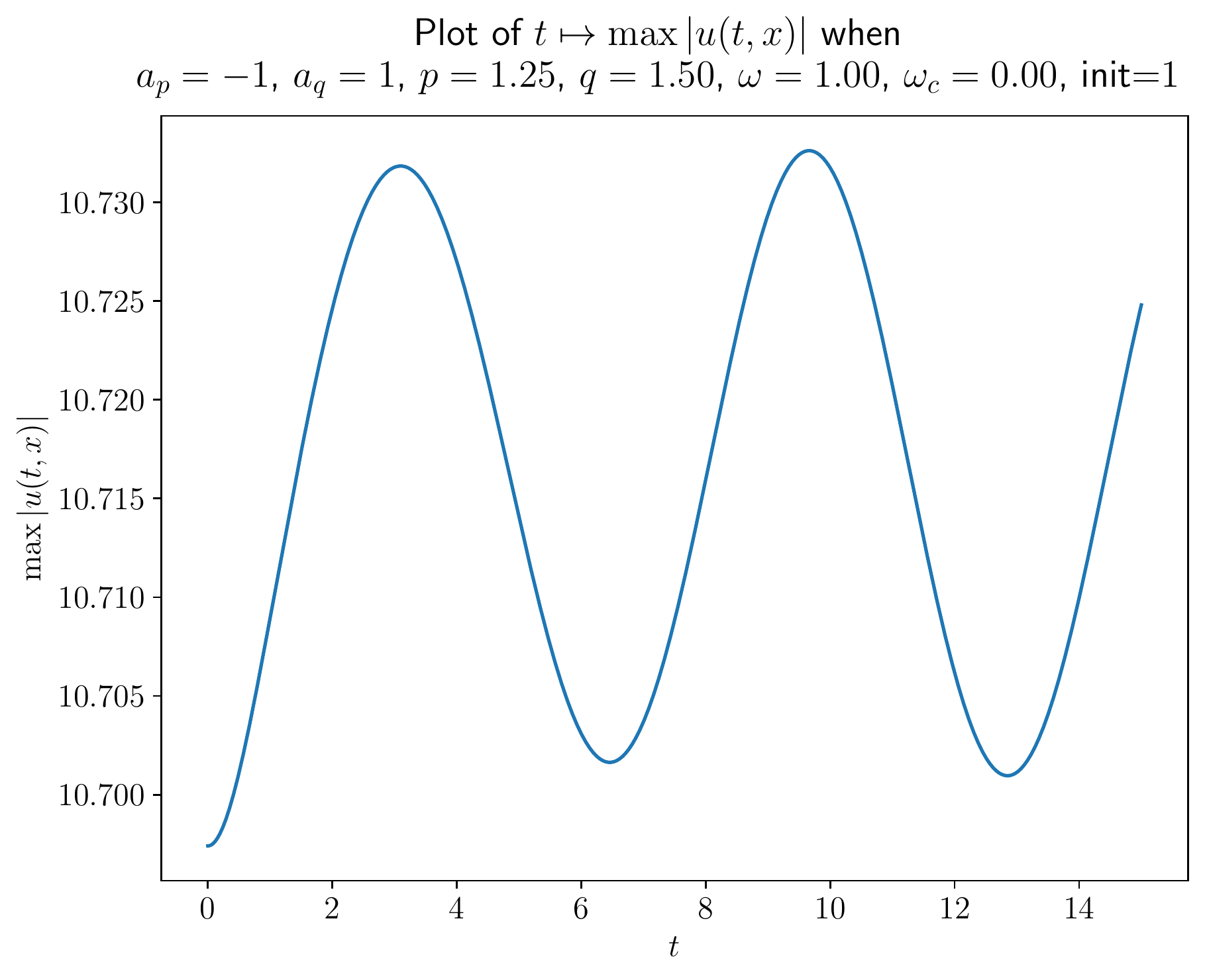}
   \caption{Example of a stable numerical solution. The initial data is $u^0=(1+\eps)\phi_\omega$, $\eps=10^{-2}$.}
  \label{fig:stable}
\end{figure}

The second behavior consists in a first phase of focusing growth of the profile, which is similar to what can be observed when instability of solitons is by blow-up (e.g.~for power-type supercritical nonlinearities. However, after a certain time, the focusing phase stops and is followed by a phase in which the solution seems to oscillate  around another profile. The size of the oscillation is decaying, but at a slow pace, and we have not run the simulation long enough to observe convergence toward a final state.  An example of such a behavior is presented in Figure \ref{fig:growth}.
 \begin{figure}[htpb!]
   \centering \includegraphics[width=.48\textwidth]{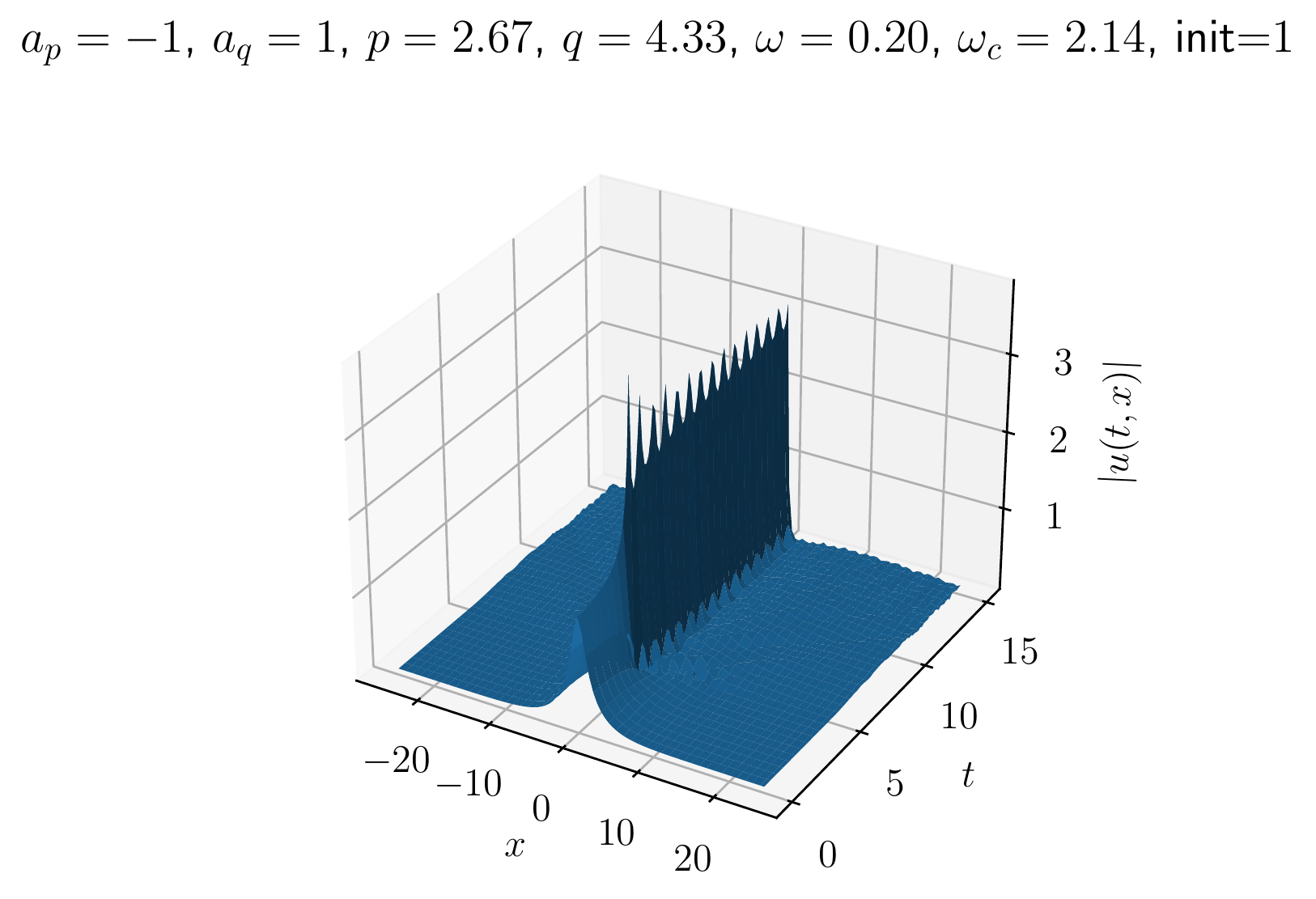}
   ~\includegraphics[width=.38\textwidth]{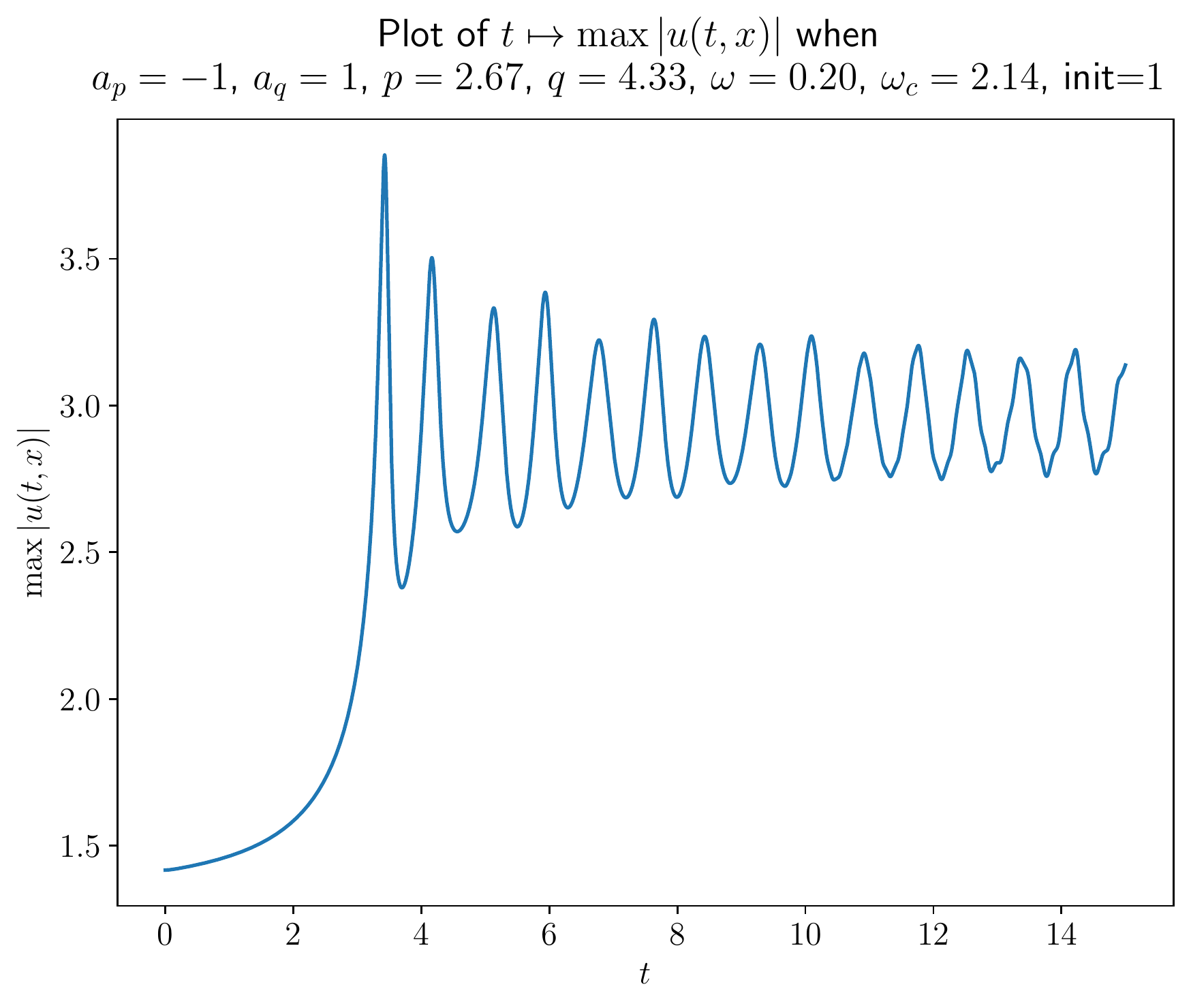}
   \caption{Example of a growing/oscillating numerical solution. The initial data is $u^0=(1+\eps)\phi_\omega$, $\eps=10^{-2}$.}
  \label{fig:growth}
\end{figure}

Finally, the third behavior that we have observed could be characterized as scattering, as the profile of the solution is simultaneously decreasing in height while spreading over the whole line. As before,  the decay is rather slow and we have not run the simulation long enough for the solution to converge to $0$.  An example of such a behavior is presented in Figure \ref{fig:scatter}. Observe that the domain of calculation is $[-50,50]$, but the solution is represented only on $[-20,20]$, which explains the non-zero values observed at the boundaries on the left figure. 
 \begin{figure}[htpb!]
   \centering \includegraphics[width=.48\textwidth]{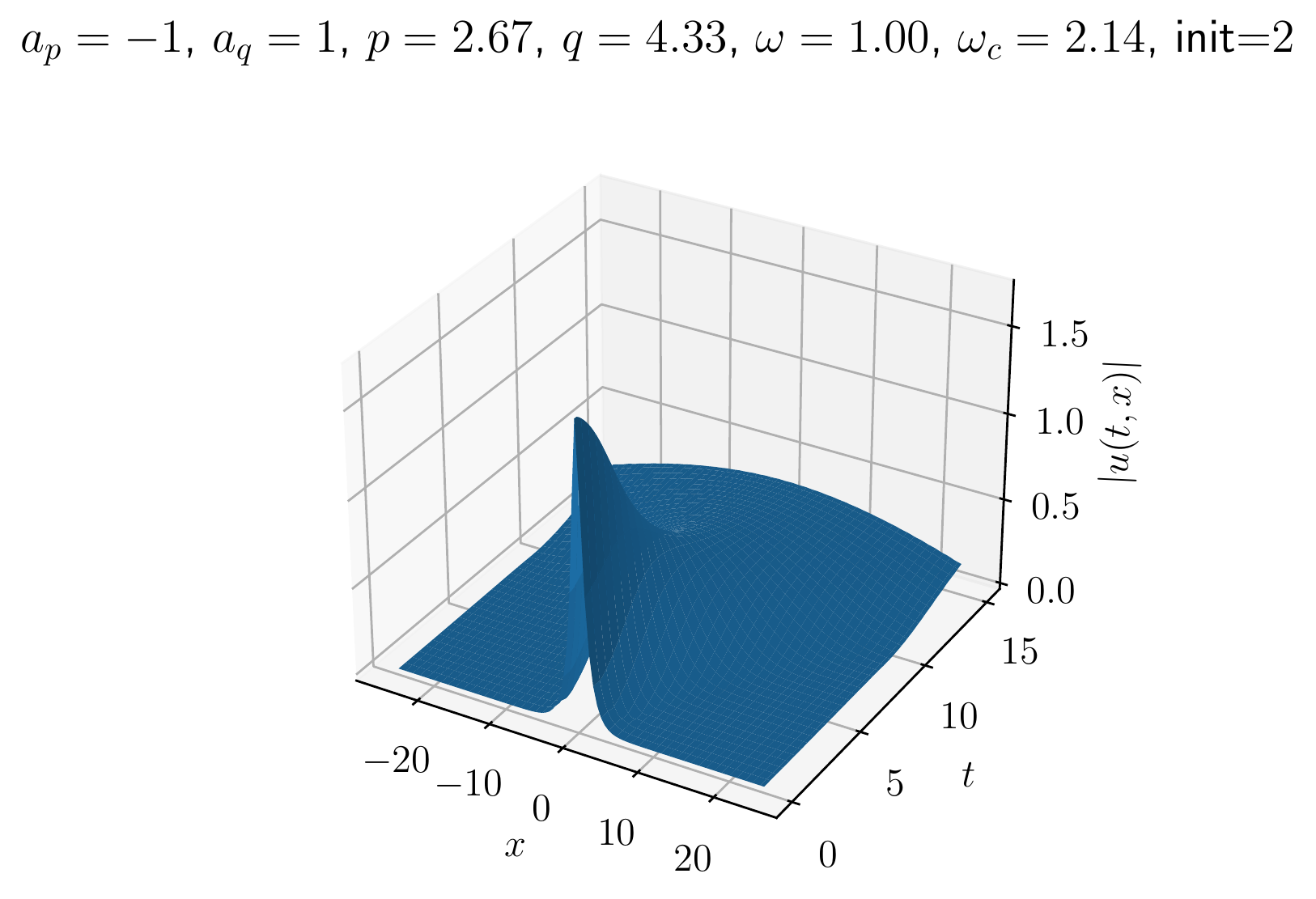}
   ~\includegraphics[width=.38\textwidth]{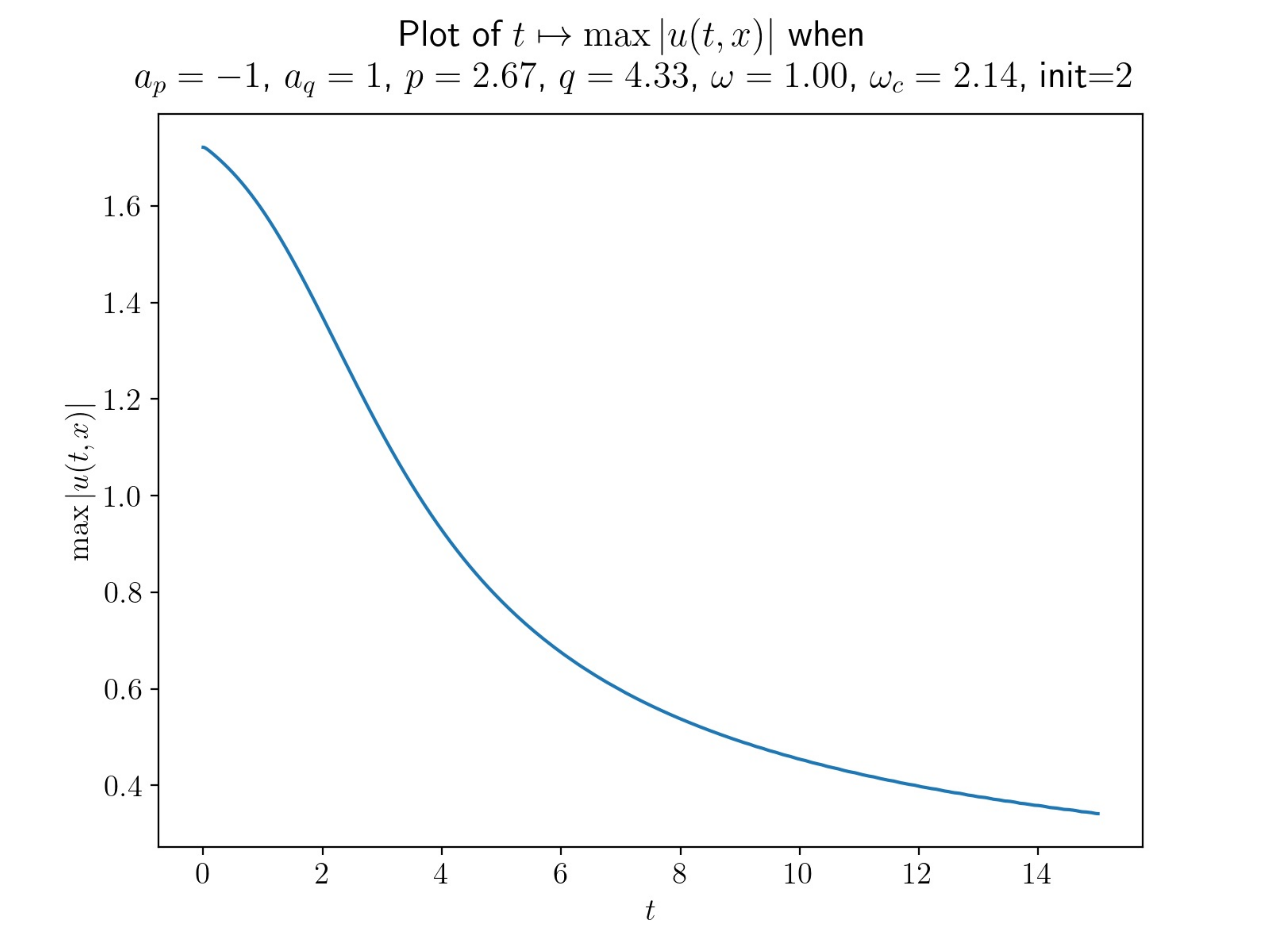}
   \caption{Example of a scattering numerical solution. The initial data is $u^0=(1-\eps)\phi_\omega$, $\eps=10^{-2}$.}
  \label{fig:scatter}
\end{figure}

\FloatBarrier

\bibliographystyle{abbrv}
\bibliography{kfoury-lecoz-tsai}

\ifx \undefined \booktitle \def \booktitle#1{{{\em #1}}} \fi\ifx \cftil
  \undefined \def \cftil#1{\~#1} \fi\ifx \undefined \cprime \def \cprime {$'$}
  \fi\ifx \undefined \flqq \def \flqq {\ifmmode \ll \else \leavevmode \raise
  0.2ex \hbox{$\scriptscriptstyle \ll $}\fi}\fi\ifx \undefined \frqq \def \frqq
  {\ifmmode \gg \else \leavevmode \raise 0.2ex \hbox{$\scriptscriptstyle \gg
  $}\fi}\fi\ifx \undefined \k \let \k = \c \fi\ifx \undefined \mathbb \def
  \mathbb #1{{\bf #1}}\fi\ifx \undefined \mathbf \def \mathbf #1{{\bf
  #1}}\fi\ifx \undefined \mathrm \def \mathrm #1{{\rm #1}}\fi\ifx \undefined
  \pkg \def \pkg #1{{{\tt #1}}} \fi\ifx \undefined \scr \let \scr = \cal
  \fi\def\cprime{$'$}
\begin{thebibliography}{10}

\bibitem{AbSt64}
M.~Abramowitz and I.~A. Stegun.
\newblock {\em Handbook of mathematical functions with formulas, graphs, and
  mathematical tables}, volume~55 of {\em National Bureau of Standards Applied
  Mathematics Series}.
\newblock U.S. Government Printing Office, Washington, D.C., 1964.

\bibitem{Ag07}
G.~Agrawal.
\newblock {\em Nonlinear fiber optics}.
\newblock Optics and Photonics. Academic Press, 2007.

\bibitem{AkAnGr99}
N.~Akhmediev, A.~Ankiewicz, and R.~Grimshaw.
\newblock Hamiltonian-versus-energy diagrams in soliton theory.
\newblock {\em Physical Review E}, 59(5):6088, 1999.

\bibitem{AnHe19}
J.~Angulo~Pava and C.~A. Hern\'{a}ndez~Melo.
\newblock On stability properties of the cubic-quintic {S}chr\"{o}dinger
  equation with {$\delta$}-point interaction.
\newblock {\em Commun. Pure Appl. Anal.}, 18(4):2093--2116, 2019.

\bibitem{AnHePl19}
J.~Angulo~Pava, C.~A. Hern\'{a}ndez~Melo, and R.~G. Plaza.
\newblock Orbital stability of standing waves for the nonlinear
  {S}chr\"{o}dinger equation with attractive delta potential and double power
  repulsive nonlinearity.
\newblock {\em J. Math. Phys.}, 60(7):071501, 23, 2019.

\bibitem{AnBaBe13}
X.~Antoine, W.~Bao, and C.~Besse.
\newblock Computational methods for the dynamics of the nonlinear
  {S}chr\"odinger/{G}ross-{P}itaevskii equations.
\newblock {\em Comput. Phys. Commun.}, 184(12):2621--2633, 2013.

\bibitem{BeFoGe20}
J.~Bellazzini, L.~Forcella, and V.~Georgiev.
\newblock Ground state energy threshold and blow-up for nls with competing
  nonlinearities.
\newblock {\em arXiv preprint arXiv:2012.10977}, 2020.

\bibitem{BeCa81}
H.~Berestycki and T.~Cazenave.
\newblock {Instabilit\'e des \'etats stationnaires dans les \'equations de
  {S}chr\"odinger et de {K}lein-{G}ordon non lin\'eaires}.
\newblock {\em C. R. Acad. Sci. Paris}, 293(9):489--492, 1981.

\bibitem{BeLi83-1}
H.~Berestycki and P.-L. Lions.
\newblock Nonlinear scalar field equations. {I}. {E}xistence of a ground state.
\newblock {\em Arch. Rational Mech. Anal.}, 82(4):313--345, 1983.

\bibitem{Be04}
C.~Besse.
\newblock {A relaxation scheme for the nonlinear Schr{\"o}dinger equation}.
\newblock {\em SIAM Journal on Numerical Analysis}, 42(3):934--952, 2004.

\bibitem{CaKlSp20}
R.~Carles, C.~Klein, and C.~Sparber.
\newblock {On soliton (in-)stability in multi-dimensional cubic-quintic
  nonlinear Schr{\"o}dinger equations}.
\newblock 21 pages, Dec. 2020.

\bibitem{Ca03}
T.~Cazenave.
\newblock {\em Semilinear {S}chr\"odinger equations}, volume~10 of {\em Courant
  Lecture Notes in Mathematics}.
\newblock New York University / Courant Institute of Mathematical Sciences, New
  York, 2003.

\bibitem{CaLi82}
T.~Cazenave and P.-L. Lions.
\newblock Orbital stability of standing waves for some nonlinear
  {S}chr\"odinger equations.
\newblock {\em Comm. Math. Phys.}, 85(4):549--561, 1982.

\bibitem{CoPe03}
A.~Comech and D.~Pelinovsky.
\newblock Purely nonlinear instability of standing waves with minimal energy.
\newblock {\em Comm. Pure Appl. Math.}, 56(11):1565--1607, 2003.

\bibitem{DeGeRo15}
S.~De~Bi\`evre, F.~Genoud, and S.~Rota~Nodari.
\newblock Orbital stability: analysis meets geometry.
\newblock In {\em Nonlinear optical and atomic systems}, volume 2146 of {\em
  Lecture Notes in Math.}, pages 147--273. Springer, Cham, 2015.

\bibitem{DeRo19}
S.~De~Bi{\`e}vre and S.~Rota~Nodari.
\newblock Orbital stability via the energy--momentum method: The case of higher
  dimensional symmetry groups.
\newblock {\em Archive for Rational Mechanics and Analysis}, 231(1):233--284,
  2019.

\bibitem{FuHa21}
N.~Fukaya and M.~Hayashi.
\newblock Instability of algebraic standing waves for nonlinear
  {S}chr\"{o}dinger equations with double power nonlinearities.
\newblock {\em Trans. Amer. Math. Soc.}, 374(2):1421--1447, 2021.

\bibitem{Fu03}
R.~Fukuizumi.
\newblock {\em Stability and instability of standing waves for nonlinear
  {S}chr\"odinger equations}.
\newblock PhD thesis, Tohoku Mathematical Publications 25, June 2003.

\bibitem{GeMaWe16}
F.~Genoud, B.~A. Malomed, and R.~M. Weish{\"a}upl.
\newblock Stable {NLS} solitons in a cubic-quintic medium with a delta-function
  potential.
\newblock {\em Nonlinear Anal.}, 133:28--50, 2016.

\bibitem{GrShSt87}
M.~Grillakis, J.~Shatah, and W.~Strauss.
\newblock Stability theory of solitary waves in the presence of symmetry. {I}.
\newblock {\em J. Funct. Anal.}, 74(1):160--197, 1987.

\bibitem{GrShSt90}
M.~Grillakis, J.~Shatah, and W.~Strauss.
\newblock Stability theory of solitary waves in the presence of symmetry. {II}.
\newblock {\em J. Func. Anal.}, 94(2):308--348, 1990.

\bibitem{Ha21}
M.~Hayashi.
\newblock Sharp thresholds for stability and instability of standing waves in a
  double power nonlinear schr\"odinger equation, 2021.
\newblock arXiv:2112.07540.

\bibitem{HaOz92}
N.~Hayashi and T.~Ozawa.
\newblock On the derivative nonlinear {S}chr\"odinger equation.
\newblock {\em Phys. D}, 55(1-2):14--36, 1992.

\bibitem{IlKi93}
I.~Iliev and K.~Kirchev.
\newblock Stability and instability of solitary waves for one-dimensional
  singular {S}chr\"odinger equations.
\newblock {\em Differential and Integral Equations}, 6:685--703, 1993.

\bibitem{KfLeTs-github}
P.~Kfoury, S.~Le~Coz, and T.-P. Tsai.
\newblock {Stability-of-standing-waves-of-the-double-power-1D-NLS}.
\newblock
  \url{https://github.com/perlakfoury/Stability-of-standing-waves-of-the-double-power-1D-NLS},
  2021.

\bibitem{LeMaRa16}
S.~Le~Coz, Y.~Martel, and P.~Rapha\"el.
\newblock {Minimal mass blow up solutions for a double power nonlinear
  Schr\"{o}dinger equation}.
\newblock {\em Rev. Mat. Iberoam.}, 32(3):795--833, 2016.

\bibitem{LeWu18}
S.~Le~Coz and Y.~Wu.
\newblock {Stability of Multisolitons for the Derivative Nonlinear Schrödinger
  Equation}.
\newblock {\em International Mathematics Research Notices},
  2018(13):4120--4170, 2018.

\bibitem{LeRo20}
M.~Lewin and S.~R. Nodari.
\newblock The double-power nonlinear schr{\"o}dinger equation and its
  generalizations: uniqueness, non-degeneracy and applications.
\newblock {\em Calculus of Variations and Partial Differential Equations},
  59(6):1--49, 2020.

\bibitem{LiTsZw21}
F.~J. {Liu}, T.-P. {Tsai}, and I.~{Zwiers}.
\newblock {Existence and stability of standing waves for one dimensional NLS
  with triple power nonlinearities}.
\newblock {\em {Nonlinear Anal., Theory Methods Appl., Ser. A, Theory
  Methods}}, 211:34, 2021.
\newblock Id/No 112409.

\bibitem{Ma08}
M.~Maeda.
\newblock Stability and instability of standing waves for 1-dimensional
  nonlinear {S}chr\"{o}dinger equation with multiple-power nonlinearity.
\newblock {\em Kodai Math. J.}, 31(2):263--271, 2008.

\bibitem{Oh95}
M.~Ohta.
\newblock Instability of standing waves for the generalized
  {D}avey-{S}tewartson system.
\newblock {\em Ann. Inst. H. Poincar\'e Phys. Th\'eor.}, 62(1):69--80, 1995.

\bibitem{OhYa15}
M.~Ohta and T.~Yamaguchi.
\newblock Strong instability of standing waves for nonlinear {S}chr\"{o}dinger
  equations with double power nonlinearity.
\newblock {\em SUT J. Math.}, 51(1):49--58, 2015.

\bibitem{Ph21}
P.~Van~Tin.
\newblock On the derivative nonlinear {S}chr\"{o}dinger equation on the half
  line with {R}obin boundary condition.
\newblock {\em J. Math. Phys.}, 62(8):Paper No. 081502, 24, 2021.

\bibitem{We83}
M.~I. Weinstein.
\newblock Nonlinear {S}chr\"odinger equations and sharp interpolation
  estimates.
\newblock {\em Comm. Math. Phys.}, 87(4):567--576, 1982/83.

\bibitem{We85}
M.~I. Weinstein.
\newblock Modulational stability of ground states of nonlinear {S}chr\"odinger
  equations.
\newblock {\em SIAM J. Math. Anal.}, 16:472--491, 1985.

\end{thebibliography}

\end{document}